\documentclass[12pt]{amsart}
\usepackage{amsmath, amssymb,latexsym}
\usepackage{verbatim}
\usepackage{mathtools}
\usepackage{color}
\usepackage[active]{srcltx}
\usepackage{graphicx}
\usepackage{hyperref}

\newcommand{\NPHI}{\ncal_{\phi_{\lambda}}}

\newcommand{\VT}{\overline{V}_{T,\epsilon}}

\newcommand{\edit}[1]{{\color{red}{$\clubsuit$#1$\clubsuit$}}}

\def\Xint#1{\mathchoice
{\XXint\displaystyle\textstyle{#1}}%
{\XXint\textstyle\scriptstyle{#1}}%
{\XXint\scriptstyle\scriptscriptstyle{#1}}%
{\XXint\scriptscriptstyle\scriptscriptstyle{#1}}%
\!\int}
\def\XXint#1#2#3{{\setbox0=\hbox{$#1{#2#3}{\int}$ }
\vcenter{\hbox{$#2#3$ }}\kern-.6\wd0}}

\def\dashint{\Xint-}

\renewcommand{\epsilon}{\varepsilon}
\renewcommand{\phi}{\varphi}

\renewcommand{\Re}{{\operatorname{Re}}}
\renewcommand{\Im}{{\operatorname{Im}}}

\newcommand{\C}{{\mathbb C}}

\newcommand{\R}{{\mathbb R}}

\newcommand{\Z}{{\mathbb Z}}

\newcommand{\acal}{\mathcal{A}}

\newcommand{\ccal}{\mathcal{C}}
\newcommand{\dcal}{\mathcal{D}}
\newcommand{\ecal}{\mathcal{E}}

\newcommand{\hcal}{\mathcal{H}}

\newcommand{\lcal}{\mathcal{L}}

\newcommand{\ncal}{\zcal}

\newcommand{\ocal}{\mathcal{O}}

\newcommand{\scal}{\mathcal{S}}

\newcommand{\zcal}{\mathcal{Z}}

\newcommand{\dbar}{\bar\partial}
\newcommand{\ddbar}{\partial\dbar}
\newcommand{\half}{\frac{1}{2}}

\newtheorem{theo}{{\sc Theorem}}[section]

\newtheorem{cor}[theo]{{\sc Corollary}}
\newtheorem{conj}[theo]{{\sc Conjecture}}
\newtheorem{lem}[theo]{{\sc Lemma}}

\newtheorem{prop}[theo]{{\sc Proposition}}

\theoremstyle{definition}
\newtheorem{defin}[theo]{Definition}

\newtheorem{rem}[theo]{Remark}

\def\Xint#1{\mathchoice
{\XXint\displaystyle\textstyle{#1}}%
{\XXint\textstyle\scriptstyle{#1}}%
{\XXint\scriptstyle\scriptscriptstyle{#1}}%
{\XXint\scriptscriptstyle\scriptscriptstyle{#1}}%
\!\int}
\def\XXint#1#2#3{{\setbox0=\hbox{$#1{#2#3}{\int}$ }
\vcenter{\hbox{$#2#3$ }}\kern-.6\wd0}}

\def\dashint{\Xint-}

\author{  Steve Zelditch}

\email{zelditch@math.northwestern.edu}

\address{Department of Mathematics, Northwestern University,
Evanston IL,  60208-2730, USA}

\thanks{Research partially supported by NSF grant DMS-1541126}

 \title{Local and global analysis of nodal sets}

\begin{document}

\maketitle

\begin{abstract}  This is a survey of recent results on nodal sets of eigenfunctions of the Laplacian on a compact Riemannian manifold. In part the techniques are `local', i.e. only assuming eigenfunctions are defined on small balls, and in part the techniques are `global', i.e. exploiting dynamics of the geodesic flow.  The local part begins with a review of doubling indices and freqeuency functions as local measures of fast or slow growth of eigenfunctions.  The pattern of boxes with maximal doubling indices plays a central role in   the  results of Logunov-Malinnokova, giving upper and lower bounds for hypersurface measures of nodal sets in the setting of general $C^{\infty}$ metrics. The proofs of both their polynomial 
upper bound and sharp lower bound are sketched. The survey continues with a global proof of the sharp upper bound for real analytic metrics (originally proved by   Donnelly-Fefferman with local arguments), using analytic continuation to Grauert tubes. Then it reviews results of Toth-Zelditch
giving sharp upper bounds on Hausdorff measures of intersections of nodal
sets with real analytic submanifolds in the real analytic setting. Last,  it goes over  lower bounds of Jung-Zelditch on numbers of nodal domains in the case of $C^{\infty}$ surfaces of non-positive curvature and concave boundary or on negatively curved `real' Riemann surfaces, which are based on ergodic properties  of the geodesic flow and eigenfunctions, and on estimates of restrictions of eigenfunctions to hypersurfaces. The last section details recent results on restrictions.

\end{abstract}
\bigskip

The  Laplacian on a  complete Riemannian manifold $(M^n,g)$ of dimension $n$ is locally given by 
  $$\Delta_g = 
\frac{1}{\sqrt{g}}\sum_{i,j=1}^m \frac{\partial}{\partial
x_i}g^{ij} \sqrt{g} \frac{\partial}{\partial x_j},$$  where    $g_{ij} = g(\frac{\partial}{\partial
x_i},\frac{\partial}{\partial x_j}) $, $[g^{ij}]$ is the inverse
matrix to $[g_{ij}]$ and $g = {\rm det} [g_{ij}].$ When $M$ is compact,
there is an orthonormal basis $\{\phi_j\}$ of
eigenfunctions,
 \begin{equation} \label{EV}
  \Delta_g \phi_j = -
 \lambda_j^2 \phi_j,\;\;\;\; \int_M  \phi_i\phi_j dV_g= \delta_{i
j}  
\end{equation}
with $0 = \lambda_0 < \lambda_1 \leq \lambda_2 \leq \cdots \uparrow \infty$
repeated according to their multiplicities. 
When $M$ has a non-empty boundary $\partial M $, one imposes boundary conditions such as 
Dirichlet $B u = u|_{\partial M} = 0$ or Neumann $B u =
\partial_{\nu} u |_{\partial M} = 0$.

This is a survey of  recent results on nodal sets of eigenfunctions
\begin{equation}\label{EP}  \Delta_g\; \phi_{\lambda} = - \lambda^2\; \phi_{\lambda} \;\;\; \end{equation} of the
Laplacian of a compact Riemannian manifold $(M,g)$ of dimension $n$.
The nodal set of $\phi_{\lambda}$ is the hypersurface
$$\ncal_{\phi_{\lambda}} = \{x \in M: \phi_{\lambda}(x) = 0\}. $$
The two problems we discuss are upper/lower bounds on  the surface measure    $\hcal^{n-1}(\ncal_{\phi_{\lambda}})$ of the nodal set and the number
of nodal domains or connected components of the nodal set.

  The motivating conjecture on surface measure is the S.T.  Yau conjecture,
\begin{conj}

Let $(M^n, g)$ be a compact $C^{\infty}$  Riemannian
manifold of dimension m, with or without boundary and let $\Delta \phi = - \lambda^2 \phi$. Then 
$$ \lambda \lesssim {\mathcal
H}^{n-1}(Z_{\phi_{\lambda}}) \lesssim \lambda.
$$
\end{conj}\noindent Here, $f \lesssim g$ means $f \leq C g$ where $C$ is independent of the eigenvalue.   In the real analytic case, the conjecture was proved by Donnelly-Fefferman in \cite{DF88}. Recently, the sharp lower bound was proved by A. Logunov \cite{LoUB16}.

 Regarding numbers of nodal domains, the problem is to find conditions on 
 $(M,g)$ ensuring the existence of a sequence of eigenfunctions $\phi_{\lambda_j}$ for which the number $N(\phi_{\lambda_j}) $ of nodal domains tends to infinity,  and to give a quantitative lower bound for the number.  The analysis brings in problems concerning `restrictions $\phi_j |_H$ of eigenfunctions' to submanifolds $H \subset M$, such  as measuring

 \begin{itemize}
 
  \item  $L^p$ norms $\int_H |\phi_j|^p d S$;  
  
   \item  `period' integrals $\int_H f \phi_j dS$, or 
   
   \item `matrix elements' $\langle Op_h(a) \phi_j |_H, \phi_j |_H \rangle_{L^2(H)}$ of restrictions, and their `quantum  limits', also known as microlocal defect measures. 
   
    \end{itemize} 
    Since there have been a lot of recent articles studying these problems for their own sake, we survey them along with their applications to nodal sets
    in Section \ref{RSECT}.
    In a rough sense, restriction problems arise because it can be simpler to study zeros of eigenfunctions on submanifolds (such as curves on a surface) than on the whole manifold.
    Comparing norms or matrix elements of restrictions $\phi_j|_H$ to the global norms or matrix elements  on  $M$ gives a refined characteristic   of the concentration and oscillation properties of sequences of eigenfunctions. 

As in \cite{Zel08} we contrast two approaches to the study of nodal sets:
\bigskip

\begin{itemize}

\item The local approach: Studying local eigenfunctions on small balls $B$, often of wave-length scale, or harmonifying local eigenfunctions on $B$  to produce harmonic functons on the cone $B \times \R_+$. The lower bound on $\hcal^{n-1}(\ncal_{\phi_{\lambda}})$ is a local problem: it suffices to obtain a lower bound in some ball $B \subset M$. The upper bound is more global since one must obtain it in all balls.

A key idea is to introduce a `local frequency' $N = N_{u}(x, \rho)$ of a harmonic function or an  eigenfuntion $u$  in a ball $B_{\rho}(x)$. It could be defined by an Almgren-type
frequency function or by a doubling estimate.   Key tools are doubling estimates or frequency function estimates, and  elliptic theory, such as Harnack inequalities, propagation of smallness, three-ball inequalities, etc.
  \bigskip

\item The global approach: wave equation methods for studying high frequency  asymptotics of global eigenfunctions. Counting nodal domains is a global problem: connectivity of the nodal set cannot be detected from nodal behavior in small balls.

The parameter $\lambda$
in \eqref{EP} is usually called the `frequency' but it is a globally defined frequency compared with the local frequency $N(x, \rho)$.  Microlocal analysis gives techniques for explointing  dynamical properties of the geodesic flow to obtain results on eigenfunctions in the ergodic or completely integrable cases. In the real analytic case, analytic continuation of the wave kernel to Grauert tubes is a useful approach to  upper bounds on growth of  nodal sets, and of intersections of nodal sets with curves or hypersurfaces. 

\end{itemize}
\bigskip

At the present time, global harmonic analysis arguments have succeeded in giving sharp upper bounds on nodal surface measure for  real analytic $(M,g)$  and in giving power law lower bounds in the $C^{\infty}$ case, but have not succeeded in giving sharp upper bounds or lower bounds in the $C^{\infty}$ case. One aim of this survey is to illustrate problems and results where global methods are applicable and at times indispensible.  Another  aim is to try  to identify the short-comings of the global methods on nodal area estimates and potential avenues of improvement.

In some sense these notes are directed to `globalists' or semi-classical analysts, who are less familiar with the local arguments. 
The survey is organized so that each section has its own bibliography. 

I thank   Bogdan Georgiev,  Alexander Logunov, Bill  Minicozzi, and Misha Sodin for explaining many aspects of the local theory of nodal sets and in particular for helping me navigate \cite{LoUB16,LoLB16}. 
I also thank R. Chang and M. Geis for comments and corrections on this exposition.

\section{Background}

This section provides background on techniques in the local study of eigenfunctions:  harmonification,
doubling exponents, frequency functions, and  wave-length rescaling. These techniques are used in the works of Donnelly-Fefferman, Fang-Hua Lin, and
Logunov-Malinnikova (see \cite{DF88,NPS,RF15}. Less background is provided on global techniques, for which we refer to \cite{Zel08,Zel17}.

Two issues are emphasized: (i) estimating the local surface measure of
nodal sets of eigenfunctions in small balls $B$ in terms of the local frequency in $B$ rather than in terms of the  global frequency $\lambda$; (ii)  partitioning of $(M,g)$ into small balls of `fast growth' and `slow growth' and the implications of the local growth properties on norms and nodal sets of eigenfunctions in small balls. We use the notations $B_r(x) = B(x,r)$ interchangeably for metric balls of radius $r$ centered at $x$.

Global arguments have always used the global frequency $\lambda$ to estimate even the local growth of eigenfunctions and nodal sets. The local frequency $N$ introduced below is a kind of local definition of an (inverse) Planck constant adapted to the eigenfunction. One of the short-comings of the global methods is that it  is not clear how to exploit the local frequency  in wave equation arguments.

\subsection{\label{HARMCONE} Harmonification of eigenfunctions}

As is well-known, eigenfunctions on the sphere $S^n$ of eigenvalue $N(N+ n -1)$ are restrictions
of homogeneous harmonic polynomials of degree $N$ on $\R^{n+1}$. This
suggests converting eigenfunctions on general Riemannian manifolds $(M,g)$ to harmonic functions on the cone over $M$.
There are two essentially equivalent ways to harmonify eigenfunctions:

\begin{itemize}

\item Form the cone $\R_+ \times M $ and consider the metric
$\hat{g} = dr^2 + r^2 g$. Let $\hat{\phi_{\lambda}} = r^{\alpha} u$ where
$$\alpha = \frac{1}{2} \left( \sqrt{4 \lambda + (n-1)^2} - (n-1)
\right). $$ Let $\hat{\Delta}$ be the Laplacian on the cone. Then,
$$\hat{\Delta} \hat{\phi_{\lambda}} = 0. $$

\item Form $\R_+ \times M$ and consider $e^{-\lambda t} \phi_{\lambda} $.
Then $(\partial_t^2 + \Delta) (e^{- \lambda  t} \phi_{\lambda} ) = 0. $

\end{itemize}
This approach was    taken in \cite{GaL86,Lin91} and used further
in \cite{Ku95}, among other places. The advantage is that harmonic functions
have useful properties that eigenfunctions lack: in particular, their
frequency functions are monotone. 

Of course, the two methods are equivalent by changing variables $r^{\alpha} = e^{\lambda t}$. We prefer the first version because it reduces to the usual
harmonic extension of spherical harmonics on $S^n$ as homogeneous harmonic polynomials on $\R^{n+1}$.

In \cite{LoUB16,LoLB16} A. Logunov harmonifies eigenfunctions to harmonic functions on $\R_+ \times M$,  and then  fixes a (macroscopic) geodesic ball
$B_g(0, R)$ for a fixed but small $R$ in this space. He writes the Laplacian $\Delta$ on the cone  in normal coordinates. Thereafter it is treated as a fixed elliptic operator with $C^{\infty}$ coefficients on a fixed ball $B$ or cube $Q$
in $\R^{n+1}$ and only harmonic functions are explicitly considered until the final sections.

One of the aims of this exposition is to re-formulate the statements and proofs so that they apply to re-scalings of eigenfunctions and Laplacians
on wave-length scale balls of the original manifold $M$ (see the next section). We record the main properties used about harmonic functions on a fixed ball $B \subset \R^{n+1}$ with respect to a fixed elliptic operator  $L$ with the aim of checking whether the properties are valid for rescaled eigenfunctions as well. We refer mainly to \cite{LoUB16}.

\begin{enumerate}

\item (Almost-)monotonicity of the frequency function (denoted $\beta(r)$ in 
\cite[p. 2]{LoUB16}) or doubling index $N(B(x, r)) = N(x, r)$. \bigskip

\item The comparison between $L^{\infty}$ and $L^2$ norms of harmonic functions on concentric geodesic spheres. \bigskip

\item Propagation of smallness: Let $L u = 0$ on the unit cube $Q$ and $q \subset \half Q$ be a sub-cube of side $r$ and let $F$ be a face of $q$. Then if
$|u| \leq \epsilon$ on $F$ and $|\nabla u| \leq \frac{\epsilon}{r}$ on $F$ then
$\sup_{\half q} |u| \leq \epsilon^{\alpha}$ for some $C > 0$ and $\alpha \in (0,1)$.\bigskip

\item The  standard elliptic $(L^{\infty}$ Bernstein) estimate $\sup_q |\nabla u| \leq C A |\sup u|$ (see \cite[Lemma 4.1]{LoUB16}.)\bigskip

\item The standard elliptic estimates in the proof of \cite[Lemma 7.1]{LoUB16} and \cite[Lemma 7.2]{LoUB16} bounding $\sup_{B(p,r)} |u|^2$ in terms of $H(r(1 \pm \epsilon))/r^{n-1}$ where $H(r) = \int_{\partial B(p,r)} u^2 d S$ (see \eqref{HDEF}). \bigskip

\item Comparison of the $L^2$ norm of a harmonic function on the boundary of a ball in terms of the $L^2$ norm in the ball \cite[Lemma 4.1]{LoLB16}.

\item Harnack inequalities: see \cite[(27)]{LoLB16} and \cite[Lemma 8.1]{LoLB16}.\bigskip

\end{enumerate}

\subsection{\label{DILATE} Small balls and local dilation}

Eigenfunctions also resemble local  harmonic functions on the original manifold $M$ when rescaled on  wavelength balls of 
radius  $r = \epsilon \lambda^{-1}$ for  $\epsilon$. Rescaling on the wave-length scale has been one of the standard approaches since the early work of L. Bers and Hartman-Wintner. Recent works with this approach include \cite{Ma08, NPS}.

Let us pull back the eigenvalue 
equation to the tangent space $T_{x_0} M$ and dilate $\phi_{\lambda}(u) \to \phi_\lambda(tu)$ in the tangent space $T_{x_0} M$. That is, we define the  dilation operators by
\begin{equation}D_{t}^{x_0} \phi_{\lambda}(u) = \phi_\lambda(\exp_{x_0} t u),  \qquad u \in T_{x_0} M. \end{equation} Write
$\Delta= \sum_{i, j = 1}^n g^{i j}(x) \frac{\partial^2 }{\partial x_i \partial x_j} + \sum_{j = 1}^n \Gamma_j(x) \frac{\partial}{\partial x_j} $ in geodessic local coordinates and define the
{\it osculating operator} at $p$ to be the constant coefficient operator
$$ \Delta^{(2)} = \sum_{i, j = 1}^n g^{i j}(0) \frac{\partial^2 }{\partial x_i \partial x_j} + \sum_{j = 1}^n \Gamma_j(0) \frac{\partial}{\partial x_j}. $$
We then rewrite the eigenvalue equation as
\begin{align}\label{RESC} \notag D_{t}^{x_0}  \Delta_g (D_t^{x_0} )^{-1} &\phi(\exp_{x_0} t u) =
\lambda^2  \phi(\exp_{x_0} t u)  \\ 
&\implies \left[t^{-2} \Delta^{(2)} + t^{-1} \Delta^{(1)} + \cdots \right]  \phi(\exp_{x_0} t u) = \lambda^2
 \phi(\exp_{x_0} t u).  \end{align}
 When $t = \lambda^{-1}$, the leading order term is the osculating equation
 $\Delta^{(2)} u = \epsilon u$.  The entire rescaled equation becomes an eigenvalue problem of small frequency for an associated  Schroedinger  equation, see \cite[(3.1)]{NPS} or \cite[(2.6)]{Ma08}.
 Rescaling 
flattens out the metric so that (for $\epsilon$ sufficiently small) it is close to the Euclidean metric on a ball of radius $1$, and changes
the eigenvalue from $\lambda$ to $\epsilon $.   For sufficiently small $\epsilon$, many properties of harmonic functions hold true for this equation,
for instance the maximum principle and some  mean value inequalities. 
 
 The question we raised above is whether the properties of harmonic
 functions $L u = 0$ used in \cite{LoUB16,LoLB16} and enumerated in the previous section also hold for the
 rescaled eigenvalue problem  \eqref{RESC} with $t = \epsilon \lambda^{-1}$. 
In the next Section \ref{FDSECT} we review well-known results on monotonicity of frequency functions of eigenfunctions on wave-length scale balls which show that (1) on the list above holds for the rescaled eigenvalue problem. All of the other properties are standard elliptic estimates which hold for the rescaled eigenvalue problem. Hence, all properties of harmonic functions used in \cite{LoUB16,LoLB16} hold as well on wave-length scale balls of radius $\frac{\epsilon}{\lambda}$ for eigenfunctions on the original manifold.

Some of the proofs of the standard elliptic estimates are based on positivity of the Dirichlet Green's function of the relevant ball, and that property requires a choice of small $\epsilon$. The   Euclidean Dirichlet Green's function $G_0(\epsilon, x, y)$ for the unit ball, i.e. the kernel of  the Green's function $(\Delta_0 + \epsilon^2)^{-1}$ of the flat Laplacian, is strictly negative.
Indeed, this is well known for $\epsilon = 0$ where the Dirichlet Green's function can be constructed from the Newtonian potential
$- \frac{1}{r^{n-2}}$ by the method of reflections. For $\epsilon $ sufficiently small, so that $\lambda \leq \mu_1(B(p, \frac{\epsilon}{\lambda}))$ (the lowest Dirichlet eigenvalue of the ball), we have  $  G_D(\lambda, x, x') < 0$. This can be seen by writing $(\Delta + \epsilon)^{-1} = \int_0^{\infty} e^{ \epsilon t} e^{t \Delta} dt$ and noting
that the integral converges if $\epsilon < \lambda_1$ (the lowest Dirichlet eigenvalue).

The conclusion (admittedly with few details provided) is that the main ingredients
of Logunov's upper bound (Propagation of Smallness, the Hyperplane Lemma and the Simplex Lemma) are valid on wave-length scale balls.
See \cite{RFG17} for more details on these ingredients.




\subsection{\label{FDSECT} Frequency functions and doubling functions}

Let  $(M, g)$ be a Riemannian manifold, let $B_r(a)$ be the ball of radius $r$ centered at $a$ and let $u$ be any function on $M$.
The frequency function $N(a, r):= N_u(a,r)$ of  $u$ is  defined by
\begin{equation} \label{FF} N(a, r) = \frac{r D(a,r)}{H(a,r)},
\end{equation} where
\begin{equation} \label{HDEF} \;\; H(a,r) = \int_{\partial
B_r(a)} u^2 d \sigma, \;\; D(a,r) = \int_{B_r(a)} |\nabla u|^2 dx.
\end{equation}
In the case of a homogeneous harmonic polynomial of degree $N$ on $\R^n$, $N(a,r) = N$ for all $(a,r)$. In general, this local frequency function is most useful a  harmonic functions and 
 is a local
measure of its `degree'   as a polynomial like function on  $B_r(a)$ and  controls the
local growth rate of $u$. Some expositions of 
 frequency functions and
their applications can be found in \cite{H,Ku95}, following the
original treatments in \cite{GaL86,GaL87,Lin91}. Note that Logunov gives
the alternative definition \cite[p. 2]{LoUB16},
\begin{equation} \label{LOGFF} N_L(p, r) = \frac{r}{2} \frac{d}{dr} \log H(p, r) = \frac{r \frac{\partial}{\partial r} \int_{\partial B(p, r)} u^2 dS_r}{2 \int_{\partial B(p,r)} u^2 dS_r}. \end{equation}
He denotes $N_L$ by $\beta$, but we use that notation below for the doubling exponent.

Frequency functions may also be defined for eigenfunctions. At
least two variations have been studied: (i) where the
eigenfunctions are converted  into harmonic
functions on the cone $\R^+ \times M$ as in \S \ref{HARMCONE};
(ii) where a frequency function adapted to eigenfunctions on $M$
is defined.

Frequency functions of eigenfunctions are  defined as follows:
 Fix a point $a \in M$ and choose geodesic normal
coordinates centered at $a$ so that $a = 0$. Put
$\mu(x) = \frac{g_{ij} x_i x_j}{|x|^2}, $
 and define
\begin{equation}\label{HAR}  D(a, r) :=   \int_{B_r} \left( g^{ij} \frac{\partial
\phi_{\lambda}}{\partial x_i} \frac{\partial \phi_{\lambda}}{\partial x_j} + \lambda^2 \phi_{\lambda}^2
\right) dV =  \int_{\partial B_r} \phi_{\lambda} \frac{\partial
\phi_{\lambda}}{\partial \nu}, \;\;\;\;\mbox{resp.}\;\; H(a, r) : = \int_{\partial
B_r} \mu \phi_{\lambda}^2,
\end{equation}

The frequency function of an eigenfunction is then defined by \eqref{FF} but using \eqref{HAR}.  A key difference to the case of harmonic functions is
that the local frequency function of an eigenfunction is only monotonic on the wave length scale. In model cases on  $\R^n$ it may be expressed in terms of Bessel functions (see \cite{Zel08}). The general theorem (Theorem 2.3 of \cite{GaL86} (see also \cite{GaL87,Lin91,H} and
 \cite{Ku95} (Th. 2.3, 2.4)) is:

\begin{theo} \label{MONO}   There exists $C > 0$ such that
$e^{C r} (N(r) + \lambda^2 + 1)$ is a non-decreasing function of
$r$ in some interval $[0, r_0(\lambda)]$.
\end{theo}

\begin{rem} D. Mangoubi apparently corrects the statement in \cite{Ma13} and shows that the exponential factor should be $e^{C r^2}$. \end{rem}

Another basic fact is that the frequency of $\phi_{\lambda}$ in $B_r(a)$ is
comparable to its frequency in $B_R(b)$ if $a, b$ are close and
$r, R$ are close. More precisely, there exists $N_0(R) \ll 1$ such
that if $N(0, 1) \leq N_0(R)$, then $\phi_{\lambda}$ does not vanish in $B_R$,
while if $N(0, 1) \geq N_0(R)$, then
\begin{equation} \label{NCOMP} N(p, \frac{1}{2}(1 - R)) \leq C \; N(0, 1), \;\;
\forall p \in B_R.
\end{equation}

\subsection{Doubling estimate, vanishing order estimate and lower bound estimate}

Closely related to frequency functions are doubling exponents, which 
are also functions $\beta(p,r) = \beta(B_r(p))$ of centers and radii of balls. 
Doubling estimates were the main tool in \cite{DF88} and that article contains a wealth of information which may not have been fully exploited as yet. More recent articles based on doubling estimates are \cite{NPS,RF15}.

Define the supnorm doubling exponent $\beta(\phi, B)$ for a ball $B$ by
$$\beta(\phi, B) = \log \frac{\sup_B |\phi|}{\sup_{\half B} |\phi|}, \; {\rm or \; more\; generally} \; \beta(\phi, B; \alpha) = \log \frac{\sup_B |\phi|}{\sup_{\alpha B} |\phi|}. $$ In place of the sup-norm one may use an $L^p$ norm. 

The doubling index and frequency function both give local growth measures 
of harmonic functions. In \cite[Lemma 1.3]{LoUB16}, Logunov gives the following comparability theorem between them.

\begin{lem}Let $L u = 0$. Then for all $\epsilon \in (0,1)$ there exists $C$ and $R$ so that for $\rho > 0, t > 2$

\begin{equation} \label{MAINEST}   \left\{ \begin{array}{ll} (a) & t^{N(x, \rho)(1 - \epsilon) - C} \leq  \frac{\sup_{B(x, t \rho)}|u|}{\sup_{B(x, \rho)} |u| } \leq t^{N(x, t \rho)(1 + \epsilon)+C} \\ \\ (b) &N(x, \rho) > N_0 \implies
 t^{N(x, \rho)(1 - \epsilon) } \leq  \frac{\sup_{B(x, t \rho)}|u|}{\sup_{B(x, \rho)} |u| },  \end{array} \right.. \end{equation}
 \end{lem}


As mentioned above, the frequency function of a global eigenfunction
may be estimated in terms of the eigenvalue. 
Donnelly-Fefferman proved that for any $C^{\infty}$ metric,
$$\beta(\phi_{\lambda}, B) \leq C \sqrt{\lambda}, \;\; \implies
\max_{B_{2R}(x)} |u| \leq e^{C \sqrt{\lambda}} \max_{B_R(x)} |u|. $$
Related results are proved using the frequency function in \cite{DF,Lin} and \cite{H} (Lemma 6.1.1):

\begin{theo} \label{DOUBLE} Let $\phi_{\lambda}$ be a global
eigenfunction of  a $C^{\infty}$ $(M, g)$
there exists $C = C(M, g)$ and $r_0$  such that for $0 < r < r_0$,
$$\frac{1}{Vol(B_{2r}(a))} \int_{B_{2r}(a)} |\phi_{\lambda}|^2
dV_g \leq e^{C \lambda} \frac{1}{Vol(B_{r}(a))} \int_{B_{r}(a)}
|\phi_{\lambda}|^2 dV_g. $$

Further,
\begin{equation} \max_{B(p, r)} |\phi_{\lambda}(x)| \leq
\left(\frac{r}{r'} \right)^{C \lambda} \max_{x \in B(p, r')}
|\phi_{\lambda}(x)|, \;\; (0 < r' < r). \end{equation}

\end{theo}

The doubling estimates imply  the vanishing order estimates. Let
$a \in M$ and suppose that $u(a) = 0$. By the vanishing order
$\nu(u, a)$ of $u$ at $a$ is meant the largest positive integer
such that $D^{\alpha} u(a) = 0$ for all $|\alpha| \leq \nu$. The vanishing order
of an eigenfunction at each zero is
of course finite since eigenfunctions cannot vanish to infinite
order without being identically zero. The following estimate is a quantitative
version of this fact.

\begin{theo} \label{VO} (see \cite{DF}; \cite{Lin} Proposition 1.2 and Corollary 1.4; and \cite{H} Theorem 2.1.8.)
 Suppose that $M$ is compact and of dimension $n$. Then there exist constants $C(n), C_2(n)$ depending only on the dimension such that
the  the vanishing order $\nu(u, a)$ of $u$ at $a \in M$ satisfies
$\nu(u, a) \leq C(n) \; N(0, 1) + C_2(n)$ for all $a \in
B_{1/4}(0)$. In the case of a global  eigenfunction, $\nu(\phi_{\lambda},
a) \leq C(M, g) \lambda.$
\end{theo}

In the case of harmonic functions, one may write $u = P_{\nu} +
\psi_{\nu}$ where $P_{\nu}$ is a homogeneous harmonic polynomial
of degree $\nu$ and where $\psi_{\nu}$ vanishes to order $\nu + 1$ at
$a$. We note that highest weight spherical harmonics $C_n (x_1 + i
x_2)^N$ on $S^2$ are examples which vanish at the maximal order of
vanishing at the poles $x_1 = x_2 = 0, x_3 = \pm 1$.

 \subsection{The frequency function of an oscillatory integral}
 
 It is useful to have model oscillatory functions on which to test calculations
 of frequency functions and doubling indices. In this section we briefly consider semi-classical oscillatory integrals,
 \begin{equation} \label{OSCINT} h^{-d/2} \int_{\R^d} e^{i \Phi(x, y)/h} a(x,y) dy, \;\; x, y \in \R^n, \; d \leq n. \end{equation}
 They are constructed to oscillate at the wave-length scale $h$, so it is natural to measure their doubling indices or frequency function on balls of
 radius $\epsilon h$.  Montonicity can only be expected on balls of radius
 $r \lesssim h.$ 
 
 The simplest cases are pure WKB functions $a(x) e^{i \Phi(x)/h}$. The appropriate definition is \eqref{HAR}  but we use the complex conjugate 
 on the second factor or else take the WKB function to be $a(x) \cos \Phi(x)/h$.
 
\begin{equation}\label{HARWKB}  D(p, r) :=  h^{-1} \int_{\partial B_r}[i a(x)^2  \frac{\partial
\Phi}{\partial \nu} + h a   \frac{\partial
a(x)}{\partial \nu}], \;\;\;\;\mbox{resp.}\;\; H(p, r) : = \int_{\partial
B_r}  |a(x)|^2  \frac{g_{ij} x_i x_j}{|x|^2},
\end{equation}
and it is evident that the frequency is of order $h^{-1}$ for balls of macroscopic size. 
 
If we set $r = \epsilon h$ then the factor $r$ in the numerator kills $h^{-1}$ and we get
 $$N(p, \epsilon h) = \epsilon\; \frac{\int_{\partial B_r}[i a^2(x)  \frac{\partial
\Phi}{\partial \nu} + h  a \frac{\partial
a(x)}{\partial \nu}] d \sigma}{\int_{\partial
B_{\epsilon h} (p)} |a(x)|^2    \frac{g_{ij} x_i x_j}{|x|^2} d \sigma}.$$

Oscillatory integrals with positive complex phases are also important since
Gaussian beams are of this kind. In this case, $u= a(x) e^{- \Phi(x)/h}$ where
$\Phi(x) \geq 0$. In this case, the `phases' do not cancel in either numerator or denominator and we get the additional factors $e^{- 2 \Phi(x)'h}$ in both:
$$N(p, r) = \frac{ h^{-1} \int_{\partial B_r}[ a(x)^2  \frac{\partial
\Phi}{\partial \nu} + h a   \frac{\partial
a(x)}{\partial \nu}]e^{- 2 \Phi(x)'h}}{ \int_{\partial
B_r} e^{- 2 \Phi(x)'h} |a(x)|^2  \frac{g_{ij} x_i x_j}{|x|^2}} .$$ 
In the case of a Gaussian beam, $\Phi = y^2$ in Fermi normal coordinates  $(s,y)$ along a stable elliptic geodesic $\gamma$ of a surface. Here, $s$ is arc-length along $\gamma$ and   $y$ is the normal distance to the geodesic.
The Gaussian beam is supported in an $\sqrt{h}$ tube around the geodesic. If $r = \epsilon h$ and $p \in \gamma$, the Gaussian factor $e^{- y^2/h}$ is essentially equal to $1$ and we get a constant frequency as before. If $r = \epsilon \sqrt{h}$, the Gaussian factor is roughly constant and again the frequency is roughly a constant. If we now center the ball at a point $p$ near the poles, far from $\gamma$, both numerator and denominator are exponentially decaying. 

\subsection{For which $B_r(p)$ are doubling exponent bounds achieved?}

For which balls $B(x_0, r)$ is the doubling bound
$$\sup_{B_{\frac{r}{2}}} |u| \leq 2^N \; \sup_{B_{\frac{r}{4}}}  |u|, \;\; N = \beta_u(x_0, r)$$
achieved?  Call them balls of rapid growth.\bigskip

Model example: $r^{\lambda}$ on $[0, T]$. On a manifold,
the model example is a ball around a point $x_0$ where an eigenfunction $\phi_{\lambda}$ achieves
the maximal VO (vanishing order) of $\lambda$, such as a Gaussian beam $C (x + i y)^N$  at the poles $x= y= 0$ of
a sphere. 

In general,  maximal VO points and maximal doubling points are very rare.

 The highest `concentration' of the nodal set occurs at singular points where $\phi(p) = d \phi(p) = 0$,
 especially at points where $\phi$ vanishes to order $\simeq \lambda. $ Then there are $\simeq \lambda$ `spokes'
 in the nodal set emanating from the singular point, and the density of the nodal set is $\lambda$ times the usual
 one near $p$.

\begin{center}
\includegraphics[scale=0.7]{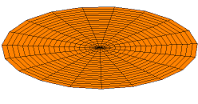}
\end{center}

Above is the picture for the unit disc of eigenfunctions with high `angular momentum' , i.e. $J_m(\rho_{m,n}) \sin m \theta, $
where $m \simeq \lambda_{m,n}$.

\subsection{\label{DIxN} Doubling indices of eigenfunctions and of their harmonifications}

The papers of Donnelly-Fefferman \cite{DF88} and Logunov \cite{LoUB16,LoLB16}  rely on  doubling index estimates for
eigenfunctions and of their harmonifications. Logunov's work in particular  relies on  doubling indices for harmonic functions on the unit ball $B_1 \subset \R^n$. In this section, we briefly compare doubling indices of  eigenfunctions  $\phi$ and of their harmonifications $u =r^d \phi$.
 
We consider product cubes $Q = [r_0- a, r_0 + a] \times
 q$ in $\R_+ \times q$,where $q$ is a cube in $M$ centered at some point $x_0$. Then  $2 Q= [r_0 - 2 a, r_0 +  2a] \times (2 q). $ The first observation is that the doubling index of 
   $u =r^d \phi$ in $Q$ is basically the sum of the doubling indices of $r^d$ in $[r_0 -a, r_0 + a]$ and that of $\phi$ in $q$. Indeed,
 $\max_Q u = \max_{[r_0 - a, r_0 + a]} r^d \times \max_q \phi =
 (a + r_0)^d \max_q \phi$ and  $$\frac{\sup_{2 Q} u}{\sup_Q u } = \left[\frac{(r_0 + 2a)}{(r_0 + a)}\right]^d
 \frac{\sup_{2 q} \phi}{\sup_q \phi }, $$
 The doubling index is the logarithm, hence  the sum.
 
 The issue is that $r^d$ has a maximal doubling index $d \simeq \lambda$ in many intervals. This depends on the point $r_0$ where the interval is centered and the radius of the interval. 
Obviously, the doubing index is $d$ if $r_0 = 0$. To emphasize the index of $\phi$ over $r^d$  it is necessary to center $r_0$ away from $0$. Logunov fixes  $Q$ to be  the unit cube, and then  it is reasonable to fix $r_0 = 1$.
Note that if the radius is of frequency scale $a = \frac{1}{d}$ and $r_0= 1$ then 
the doubling index is $d \log \frac{(1+ \frac{2}{d})}{(1 + \frac{1}{d})}
= 1. $

  A few elementary calculations
show that for  $r_0 \simeq 1$, the doubling index of $r^d$ is only `small' (i.e. independent of $\lambda$) on wave-length radius intervals $a \simeq \frac{1}{d}$  centered at $r_0$. 
 Since  $\log \frac{(2a + r_0)^d}{(a + r_0)^d} = d [\log \frac{r_0 + 2a}{r_0 + a}]$, when $r_0 = 1$  one needs $\log \frac{1 + 2a}{1 + a} \simeq \frac{1}{d}$ or $a \simeq \frac{1}{d}.$

 \subsection{Estimates of nodal sets in terms of the frequency}
 
 Local estimates of Hausdorff measures of nodal sets of a harmonic function or eigenfunction $u$  are often based on a study of the ratio,
 \begin{equation} \label{Frhox} F_{u}(x, \rho): =  \frac{\hcal^{n-1}\left(\{u = 0\} \cap B_g(x, \rho)\right)}{\rho^{n-1}}, \;\; (x \in M, \rho > 0). \end{equation}
 
 First consider upper bounds on \eqref{Frhox} in terms of the local frequency. 
 An early upper bound  is given by Hardt-Simon in  \cite[Theorem 1.7]{HS89}. 
 They denote the doubling exponent by $d$, so that for small $R$,
$||u||_R \leq 2^{d + 1} ||u||_{R/2}. $  They assume that
$$\delta(R) \leq \epsilon^{3d} ,\;\;\delta(\rho) := \sigma(\rho) + \mu_1 \rho + \mu_2 \rho^2,$$
where  $\sigma$ and $\mu_1$ involve  the coefficients of $\Delta$ and
$\mu_2 = \lambda^2$

\begin{theo} Let $g$ be a $C^{\infty}$ metric. If $x_0 \in u^{-1}(0)$ and if $\rho_0$ is small enough so that
for $R = \epsilon^{-1} \rho_0$, $\epsilon =d^{-(2n + 3)}) \epsilon_0$,
$\delta(R) \leq \epsilon^{3d}$ then for $\rho \leq \rho_0$, 
$$F_u(x, \rho) =\frac{\hcal^{n-1}(B_{\rho}(x_0) \cap u^{-1}(0))}{\rho^{n-1}} \leq C d. $$
\end{theo} In the case of eigenfunctions, the estimate is only proved for very small balls:  their assumptions imply that
$\rho_0 = C \lambda_j^{- C \sqrt{\lambda_j}} $ (See  \cite[p. 520]{HS89} ).

 Much sharper upper bounds are proved in the case of real analytic metrics in  Section 3 of \cite{Lin91}.  In \cite[Theorem 3.1]{Lin91} is proved
  \begin{theo} Let $g$ be an analytic metric on the unit ball $B(0,1) \subset \R^n$. Let $u$ be a solution of $\Delta_g u = - \lambda^2 u$ in $B$. Then 
 $$\hcal^{n-1} (\ncal_u \cap B) \leq C(n,g)\; N(0,1). $$
 \end{theo}
 Lin uses the relation of the frequency function to doubling estimates and the reduction to harmonic functions.  In \cite[Theorem 3.1']{Lin91} is proved:
 \begin{theo} Let $g$ be an analytic metric on the unit ball $B(0,1) \subset \R^n$. Let $u$ be a solution of $\Delta_g u = 0$ in $B$. Then 
 $$\hcal^{n-1} (\ncal_u \cap B) \leq C(n,g)\; N(0,1). $$
 \end{theo}
 Lin also proves that $\#\{z \in B_{\half}: f(z) = 0\} \leq C_0 N$ for a non-zero complex analytic function in a disc \cite[Lemma 3.2]{Lin91}. In the end, the nodal estimate for eigenfunctions is reduced by an integral geometry argument (Crofton formula) as in \cite{DF88} to the case of complex analytic functions of one complex variable. However, there is a gain in that the estimate is in terms of the frequency function rather than the eigenvalue.
 
 Next cosider lower bounds on \eqref{Frhox}.
 In  \cite[Lemma 3.1]{LoM16}, Logunov-Malinnikova use (and prove) the following lower bound, 
   \begin{prop}\label{MLLB} If $L u = 0$  on $B$ and if $\sup_{B_{\frac{r}{2}}} |u| \leq 2^N \sup_{B_{\frac{r}{4}}} |u|$,
  then $$\hcal^{n-1}(\{u = 0\} \cap \{|x| \leq r/2 \} \geq c r^{n-1} N^{2 - n}. $$
    If  $u(0) = 0$ and $\max_{B_1} |u| \leq 2^N \max_{B_{\half}} |u|$, then 
   \begin{equation}\label{Nn-2} \hcal^{n-1} \{x: |x| \leq 1, \; u(x) = 0\} \geq C N^{2 - n}. \end{equation}
     \end{prop}    In \cite[(30), Theorem 7.1]{LoLB16}  a weaker bound is quoted:
   nodal volumes of harmonic functions on $B \subset \R^n$:
   \begin{equation} \label{LoLBN} \frac{\hcal^{n-1}(\{u = 0\} \cap B_g(x, \rho)}{\rho^{n-1}} > \frac{C}{N^{n-1}}. \end{equation} 
  An important point is that the lower bound is better for small $N$ than for large $N$.  
  

By comparison, the previously best  lower bounds on nodal volumes \cite{CM,SoZ12} are in terms of the global frequency:
\begin{equation}\label{1}
 \hcal^{n-1} (\ncal_{\phi_{\lambda}}) \geq C_g\; \lambda^{1-\frac{n-1}2}
\end{equation}
This is only an improvement on \eqref{LoLBN}-\eqref{1} when $N \simeq \lambda$. A key point in both \cite{DF88} and \cite{LoUB16,LoLB16} is that
$N(x, r)$ is substantially smaller than $\lambda$ for most wave-length scale balls $B(x, \frac{C}{\lambda})$.

 \subsection{Cubes/balls of fast growth and slow growth}
 
 As discussed in \cite[p. 165]{DF88}, if one covers $M$ by  wave-length scale balls
 $B_{\nu} = B(x_{\nu}, \frac{C}{\lambda_j})$ in an efficient way, then on roughly half of the balls one has the estimate
 \begin{equation} \label{BDDE} \int_{Q_{\nu}} |\phi_j|^2 \leq C \; \int_{B_{\nu}} |\phi_j|^2 \end{equation}
 where $Q_{\nu}$  is a cube containing the double of $B_{\nu}$. In other words, the doubling exponent or frequency function of half of the balls is bounded by a constant indpendent of the eigenvalue. This estimate implies
 that the $||\phi_j||_{L^2(B_{\nu})} \leq C ||\phi_j||_{L^1(B_{\nu})}$ and in fact that all $L^p$-norms of $\phi_j$ on $B_{\nu}$ are equivalent. See \cite[Lemma 7.6]{DF88} for the precise statement. There always exists a point $p \in B_{\nu}$ such that $\phi_j(p) = 0$. If the ball is centered at such a point then $\int_{B_{\nu}} \phi_j dV =0$. When the $L^2$ norm is equivalent to the $L^1$ norm, the positive and negative sets of $\phi_j$ in $B$ are roughly of equal volume and the isoperimetric inequality gives a lower bound on the nodal surface volume in $B_{\nu}$. 
 Thus, the bounded doubling estimate on such balls implies a lower bound for the nodal volume in such balls,
 \begin{equation} \eqref{BDDE} \implies \hcal^{n-1}(B_{\nu} \cap \ncal_{\phi_j}) \geq C \lambda_j^{-(n-1)}. \end{equation}
 See \cite[p. 182]{DF88} for the proof. 
 The proof of \eqref{BDDE} in \cite{DF88} is based on an analysis of holomorhpic functions satisfying growth estimates.

 In \cite[Section 6]{DF90c}, Donnelly-Fefferman introduce a grid of cubes and divide them into cubes of `rapid growth' and cubes of `slow growth'.  Cubes of rapid growth are defined by an $L^2$ inequality over certain annuli  \cite[(4.6)]{DF90c} (see also
 \cite[Proposition 5.4]{DF90c}). In  \cite[Lemma 6.1]{DF90c}  they prove
 \begin{lem} \label{DFLEM} Let $G(z) = \phi_j(C \lambda^{-\half} z)$ and view $G$ as defined on a Euclidean ball $B(0, 3)$. Divide the cube of side $\frac{1}{60}$ into a grid of squares of side $\delta < C_1 \lambda^{-1}$. Then, there are at most $c_3 \lambda^2$ squares of side $\delta$ where $G$ has rapid growth. \end{lem}
 
 To prove the Lemma, they  introduce a ``process of controlled bisection'' in which squares of rapid growth are bisected to separate subcubes of rapid growth and slow growth.

 \subsection{\label{CONFIGSECT} Configurations of wave-length scale balls of high 
 doubling index}

One can study configurations of points and radii of `maximal doubling index' .
For clarity, we consider wave-length scale balls $B(p, \frac{\epsilon}{\lambda})$ and ask for the configuration of centers where $N(p, \frac{\epsilon}{\lambda})$ has a given order of magnitude. 

In \cite{RF15} it is proved that the mean value of the doubling index on balls of wave-length radius (integrated with respect to the centers $p \in M$) is of order $O(1)$. Thus, almost all points are centers of balls where the doubling index is independent of $\lambda$.

Consider the case of spherical harmonics on $S^2$, which exhibit all varieties of doubling index behavior.  Gaussian beams (highest weight spherical harmonics) of degree $N$ are examples with `many' points of maximal doubling index $\lambda \simeq N$.  In polar coordinates $(r, \theta)$ centered at the pole, the Gaussian 
beam is given by $C_N (\sin r)^N \cos N \theta$. The doubling index is essentially that of the factor $(\sin r)^N$, which  for small $r$ is similar to the model case of 
$r^N$ discussed in Section \ref{DIxN}. It doubles quickly on wave-length balls  (scale $N$) except for centers contained in  an $N^{-\half}$-tube around the equator $r = \frac{\pi}{2}$.  Standard spherical harmonics $\Re Y^m_N$ behave in a similar way when $\frac{m}{N} \simeq \tau > 0$.

 On the other hand,  zonal  spherical harmonics $\Re Y^0_N$, which are rotationally invariant under $x_3$-axis rotations,  have small doubling indices everywhere, i.e. bounded independently of $N$. There is no ball on which they have `exponential growth' since there is no wave-length scale  ball in which they are exponentially small (of size $e^{- C \lambda}$). They are essentially Legendre polynomials $P_N(\cos r)$.
 
 On the other hand, the lengths of the nodal lines of $\Re Y^N_N$ and
 $\Re Y^0_N$ are both of size $ N$, as is of course insured by the Donnelly-Fefferman theorem. Conclusion: one cannot rely on a preponderence of fast growth balls to obtain lower bounds on nodal set volumes. None may exist. It is necessary to obtain lower bounds for nodal volumes in balls of slow growth. Regarding upper bounds,  balls of fast growth may have `more nodal set', and that makes them dangerous for upper bounds.
 Logunov proves that few such balls exist.

\section{Results of Logunov-Malinnikova on Yau's conjecture}

In \cite{LoLB16} A. Logunov proved Yau's lower bound conjecture for $C^{\infty}$ metrics and in \cite{LoUB16} he
gave a non-sharp  polynomial upper bound. 

\begin{theo}

Let $(M^n, g)$ be a compact $C^{\infty}$  Riemannian
manifold of dimension m without boundary and let $\Delta \phi = - \lambda^2 \phi$. Then 
there exist $c_1 >0$ and $\alpha $ such that
$$ \lambda \lesssim {\mathcal
H}^{n-1}(Z_{\phi_{\lambda}}) \lesssim  \lambda^{\alpha}.
$$
\end{theo}

The value of $\alpha$ is quite large and not stated explicitly. For any $\alpha$ produced by the Logunov method,
 Hezari has improved the result by a log factor
in negative curvature \cite{He16}. Until now, the best upper bound had been $\lambda^{\lambda}$ due to Hardt-Simon.

Although the lower bound is the more significant result, we mainly review  the upper bound in this survey since it is a simpler result which illustrates the basic techniques. Roughly speaking, the main idea is to use the Donnelly-Fefferman doubling estimates on balls $B_r*p)$ and the relations of the doubling estimates as $p, r$ vary.  Longunov-Malinnikova refer to these arguments as `combinatorial'. Logunov's work emphasizes  the  `distribution of doubling
indices.'

  Logunov `harmonifies' eigenfunctions to harmonic functions
  on the cone $M \times \R_+$ and thereafter prove results on harmonic functions
  on the cone.
  For those primarily interested in eigenfunctions, harmonification may seem to  obscure the `geometry of eigenfunctions' (see Section \ref{FDSECT}), and one may prefer the alternative approach of rescaling eigenfunctions on wave-length balls to produce almost-harmonic functions on balls of $M$.   This is a standard approach  used in \cite{LoM16} but not
  in \cite{LoUB16} or \cite{LoLB16}.  In this article, we check each step to see if and how it adaptes  to wave-length scaled eigenfunctions.

\subsection{Sketch of Logunov's upper bound on nodal hypersurface volumes}

The proof is based on the combinatorial structure of the set where the
doubling index of $u$ is near maximal. By combinatorial structure is meant the configuration of subcubes $q$ of a unit cube $Q$ of fixed sidelength 
$L(Q)/A$. 

The main ingredient of \cite{LoUB16} is a quantitative bound on the set where
the doubling index in a macroscopic cube $Q$ is near maximal. Roughly speaking, it shows that the maximal doubling exponent set is of codimension $2$. More precisely, if one partitions $Q$ into small subcubes, then the
number of subcubes with near maximal doubling exponent is $\leq \half A^{n-1}$. In the other articles, Logunov refines the number to $\epsilon A^{n-1}$ and that is why we say it is morally of codimension two, much
as the singular set of $u$ is of codimension two.

\begin{theo} \label{MAININPUT} There exist constants $c > 0$, an integer $A$ depending
only on the dimension $n$ and $N_0 =N_0(M,g, 0), r = r(M, g, 0)$ so that for any cube $Q \subset B(0, r)$: if $Q$ is partitioned into $A^n$ equal subcubes
$q$, then the number of subcubes with doubling index $N(q) \geq \max\{\frac{N(Q)}{1 +c)}, N_0\} $ is $\leq \half A^{n-1}. $

\end{theo}

In the course of proving Theorem \ref{MAININPUT}, Loguov establishes
two results on the distribution of small cubes where the doubling exponent is near maximal.  Roughly speaking, the results say
that the small cubes of near-maximal doubling exponent lie in a thin tube around a hyperplane. This `interpretation' is not presented explicitly in \cite{LoUB16}, so we take some care to justify it. 

This statement is proved in two steps: (i) A simplex Lemma and (ii) a hyperplane Lemma. Both Lemmas pertain to harmonic functions (with respect to some Laplace operator) in a unit ball or cube.

The simplex Lemma says: If there are $(n+1)$ points in $\R^n$ where the
near-maximal doubling index is achieved, then the doubling index is larger by a factor $(1 + \epsilon)$ at the barycenter of the convex hall of the points. The moral is that the simplex $S$ has to be quite flat in the sense
that its relative width $w(S): = \rm{width}(S)/\rm{diam}(S)$ is small, where
$\rm{width}(S)$ is the minimum distance between a pair of hyperplanes enclosing $S$.

The hyperplane Lemma (or rather its Corollary) states that if one partitions
a unit cube $Q$ into $A^n$ equal subcubes $q$ and if the doubling index is almost-maximal along the bottom `row' (or hyperplane) then it is bigger somewhere in $Q$ by a factor of $2$. The moral is that the doubling index cannot be near maximal along the full bottom row 

Putting together the simplex Lemma and hyperplane Lemma, one finds that
the set of small cubes $q$ where the near-maximal doubling index is obtained must be codimension one (a tube around a hyperplane) but cannot contain many cubes in that tube, say only $\half$ of them at most.

The reader may imagine why the Lemmas are true: if the set of cubes of near 
maximal doubling index forms a `fat simplex' then one gets doubling along a  chain of cubes and eventually too high a growth rate compared to the uniform  doubling index $N(Q)$ of the large cube. The growth rate is increased if one has the full dimension $n$ of the set of near-maximal cubes, by considering chains of cubes in $n$ dimensions. This forces the chain of near-maximal cubes to lie along a hyperplane and not to contain consecutive cubes.

\subsection{Simplex Lemma}

\begin{lem} \label{SL} (Simplex Lemma for harmonic functions)  Let $\{x_{j}\}$ be vertices of a simplex in
$\R^n$ and let $B_i = B(x_i, r_i)$ where $r_i \leq \frac{K}{2} \rm{diam}(S)$,
where $K$ is from the Euclidean geometry lemma. Then there exist
$c(a, n), C(a,n) \geq K, r = r(M, g, 0,a), N_0 = N_0(M, g, O, a)$ such
that  $$\rm{if }\; S \subset B(0, r),  \rm{and\; if} \;N(B_i) \geq N, \forall i, $$ where
$N > N_0$ then $$N(x_0, C \rm{diam} (S)) \geq N(1 + c). $$
\end{lem}

The proof only uses properties of the doubling index, such as \eqref{MAINEST}
 and monotonicity to show that if  {\it $x_0$ be the barycenter
of the simplex} then,
$$ \frac{\sup_{B(x_0, t \rho(1 + \delta))}|u|}{\sup_{B(x_0, \rho(1 + c_1))} |u| } 
\geq  \frac{\sup_{B(x_i,  t \rho)}|u|}{\sup_{B(x_0, \rho(1 + c_1))} |u| } .$$

As discussed in Section \ref{DILATE}, we would like to understand the implications of the simplex Lemma for eigenfunctions on the original manifold $M$. If we harmonify the eigenfunctions and apply the simplex Lemma \ref{SL}, we get a simplex in the cone over $M$ and its barycenter with a bigger doubling index. It is necessary to radially project the simplex back to $M$ to derive implications for the eigenfunction. But we argued in Section \ref{DILATE} that the proof of Lemma \ref{SL} applies to the orginal eigenfunction in a wave-length scale ball around any point. If $\dim M =  n$ and the eigenfunction is large at $n +1$ points of the wavelength scale ball,
then it is larger by a factor $(1 + \epsilon)$ at the barycenter of  the Euclidean convex hull of the points. This statement seems to be rather weak since the frequency or doubling index should be almost constant in a wave-length scale ball.

\begin{rem} A second question involves iteration of the simplex Lemma. Once one has a configuration of $n + 1$ points with high doubling index, the Lemma gives
a new point, the barycenter of the convex hall, with a bigger doubling index. Taking this new point and $n$ of the previous ones gives yet another point, and so on. What is the limit configuration of such points of the harmonification? For eigenfunctions, what is the radial projection of
the configuration to $M$? Do the points fill out a wave-length scale ball?
(Compare \eqref{NCOMP} and the remarks above it).

\end{rem}

\subsection{Propagation of smallness} Propagation of smallness refers to the Cauchy problem for the harmonic equation $L u = 0$  and for related elliptic equations including the scaled eigenvalue problem. The Cauchy data of $u$ along a hypersurface $H$ is the pair $CD_H(u) = (u |_H, \nabla u |_H)$.  Unique continuation theorems show that if $CD_H(u) = 0$ then $u = 0$. 
Propagation of smallness gives a quantitative unique continuation theorem estimating the size of the solution in terms of the size of the Cauchy data on a hyperssurface. A general discussion is given in \cite{ARRV09} and an estimate when $H $ is the boundary of a domain is given in  \cite[Theorem 5.1]{ARRV09}.

The following is \cite[Lemma 4.3]{Lin91}

\begin{lem}Suppose that $||u||_{L^2(B_+)} \leq 1$. Suppose that
$$||u||_{H^1(\Gamma)} + ||\partial_{x_n} u||_{L^2(\Gamma)} \leq \epsilon << 1. $$
Then,$$||u||_{L2(B^+_{\half})} \leq C \epsilon^{\alpha} $$
where $C, \alpha$ depend only on $\lambda$.

\end{lem}

Logunov uses the following version, which follows from the above Lemma of Lin or \cite[Theorem 1]{ARRV09}:  Let   $g$ be a $C^{\infty}$ metric  on the unit cube $Q$ in $\R^n$
and let  $L u= 0$. Let $q \subset \half Q$ be a cube of side $r$ and $F$ a face of $q$. 

\begin{theo} Suppose that $|u| \leq 1$ in $q$. Then there exists $C > 0$ and
$\alpha \in (0,1)$ depending only on $L$ so that if $\epsilon < 1 $ and  $$\left\{ \begin{array}{ll} |u| < \epsilon & \rm{on }\; F, \\ &\\
|\nabla u| \leq \frac{\epsilon}{r}, & \rm{on} \; F \end{array} \right. $$
then $$\sup_{\half q} |u| \leq C \epsilon^{\alpha}. $$
\end{theo}

\subsection{Hyperplane Lemma}
Let $Q$ be the unit cube.  Define the uniform doubling index on $Q$ by  $N(Q) = \sup\{ N(x, r): x \in Q, r \in (0, \rm{diam}(Q)\}$.

Then let  $\{x_n = 0\}$ be a hyperplane. 
The hyperplane Lemma states
\begin{lem}  Let $Q = [-R,R]^n$ and divide $Q$ into $(2 A+1)^n$ equal
subcubes $q_i$ of sidelength $\frac{2 R}{2 A + 1}$. Let $q_{i, 0}$ be those among
$(2 A + 1)^n$ equal subcubes of $Q = [-R,R]^n$ that intersect 
$\{x_n = 0\}$. Suppose that for each $q_{i,0}$ there exists $x_i \in q_{i, 0}$
and $r_i < 10 \; \rm{diam}(q_{i, 0})$ such that $N(x_i, r_i) > N$. Then 
there exist $A_0, R_0, N_0$ so that if $A > A_0, N > N_0, R < R_0$
then $N(Q) > 2N$. \end{lem}


\begin{proof} We briefly sketch the proof. The first  step is to show that 
\begin{equation} |u|, |\nabla u|  \leq M e^{- 2 c_1 N \log A} \; \rm{on}\;
\frac{1}{8} B \cap \{x_n = 0\}. \end{equation}
Then assume $q$ has a face on $\{x_n = 0\}$.  Let $v = \frac{u}{M}$ and
apply Propagation of Smallness to $v$ to get \begin{equation}\label{PS}
 \sup_{\half q} |u|
\leq M \epsilon^{\alpha} = M 2^{- \alpha c_1 N \log A}. \end{equation}
Let $p$ be the center of $q$. The next step is to prove that 
$$\frac{\sup_{B(p, \half)} |u|}{\sup_{B(p, \frac{1}{64 \sqrt{n}}} |u|}
\geq 2^{\alpha c_1 N \log A} 
$$ 
Let $\tilde{N}$ be the doubling index for $B(p, \half)$. Then 
\begin{equation} \label{qtoQ} \frac{\sup_{B(p, \half)} |u|}{\sup_{B(p, \frac{1}{64 \sqrt{n}}} |u|}
\leq (64 \sqrt{n})^{\tilde{N}/2}.  \end{equation}

It follows that $\tilde{N} \geq c_2 N \log A$.

\end{proof} 

\begin{cor} If $N(Q) \leq N$ Then for all $\epsilon >0$, there exists
an odd integer $A_1$ so that if one divides $Q$ into $A_1^n$ equal 
subcubes $q_j$ then the number of subcubes $q_{i, 0}$  which have doubling
index $> N/2$ is $\leq \epsilon A_1^{n-1}. $
\end{cor}

 If the doubling index for every cube on the hyperplane is large,
then the doubling index on the full Q is twice as large.

These results pertain to the distribution of doubling indices.

\subsection{Logunov's upper bound on nodal growth }
Let
\begin{equation} F(N) = \sup \frac{\hcal^{n-1}(\{u = 0\} \cap Q}{\rm{diam}^{n-1}(Q)}, \end{equation}
where the $\sup$ is taken over all harmonic functions $u \in \rm{Harm}(M)$ and
cubes $Q \subset B(0, r)$ such that 
$$\{u \in \rm{Harm}(M), \; N_u(Q) \leq N\}. $$
Thus, $F(N) = \sup_{u, r, p} F_u(p, r) $ \eqref{Frhox} where $N_u(p,r) \leq N$. We can define a similar function on wave-length scale balls in the case of eigenfunctions.

\begin{lem} \label{FNLEM} There exists $\alpha, C > 0$ so that
$F(N) \leq C N^{\alpha}$. \end{lem}

The proof does not use any properties of nodal sets per se, just the obvious monotonicity and additivity properties of $\hcal^{n-1}(\ncal \cap B)$ as $B$ varies.

Call $N \in \R$ {\it bad} (with respect to $(A, c)$) if
\begin{equation}\label{FNc}  F(N) > 4 A \cdot F(\frac{N}{1 + c}). \end{equation}
This is related to the condition that $F(N)$ be  of regular growth (or variation), which are conditions on the limit function  $F^*(\tau): = \limsup_{N \to \infty} \frac{ F(\tau N)}{F(N)} < \infty, \; \forall \tau > 0. $ Logunov shows that the set of
bad $N$ is bounded so that $F^*(\tau) \leq 4 A$. This implies that $F(N)$ is regularly varying, hence of polynomial growth $N^{\alpha}$. If so, 
\eqref{FNc} implies that $\alpha \leq \frac{\log (4 A)}{\log (1 + c)}$ and depends only on the dimension.

Below \cite[(14)]{LoUB16} is
\begin{equation} \label{NINEQ} \begin{array}{lll} \hcal^{n-1}(\{u = 0\} \cap Q)
& \leq &\sum_{Q_i \in G_1} \hcal^{n-1}(\{u = 0\} \cap Q_i) +
\sum_{Q_i \in G_2} \hcal^{n-1}(\{u = 0\} \cap Q) \\ &&\\
& \leq & |G_1| F(N) \frac{\rm{diam}^{n-1}(Q)}{A^{n-1}} \; +
|G_2| F(\frac{N}{1 + c)} \frac{\rm{diam}^{n-1}(Q)}{A^{n-1}}\\&&\\
&\leq & \half F(N) \rm{diam}^{n-1}(Q) + \frac{1}{4} F(N)  \rm{diam}^{n-1}(Q),  \end{array} \end{equation}
since 
$$|G_1| \leq \half A^{n-1},\;\; II \leq |G_2| \frac{F(N) \rm{diam}^{n-1}(Q)}{4 A
A^{n-1}}, \;\;\ |G_2| \leq A^n. $$

\subsection{Some hints on the lower bound}

For the lower bound one studies the function,
\begin{equation} \label{FNINF} F(N) = \inf \frac{\hcal^{n-1}(\{u = 0\} \cap Q}{\rm{diam}^{n-1}(Q)}, \end{equation}
where the $\inf$ is taken over all harmonic functions $u \in \rm{Harm}(M)$ and
cubes $Q \subset B(0, r)$ such that 
$$\{u \in \rm{Harm}(M), \; N_u(Q) \leq N\}. $$
It is sufficient to prove that $F(N) \geq C > 0$ for all $N$.

For small $N$ one uses Theorem \ref{MLLB}. The problem is that this estimate is not good for large $N$. The idea then is to show that the (almost) minimzers of \eqref{FNINF} cannot have a large doubling index.  Suppose that $(u, Q)$ almost minimizes \eqref{FNINF} or at least that \eqref{Frhox} is $\leq 2N$. 
In   \cite[Corollary 6.4]{LoLB16} it is proved that if  $L u = 0$, then  for sufficiently large $N$, there exist at least $ [\sqrt{N}]^{n-1} 2^{C \log N/\log \log N}$ disjoint
balls of radius $\frac{1}{A} := \frac{r}{\sqrt{N} \log^6 N}$ centered at points $x_i$ where
$u(x_i) = 0$. Take the cube $Q$ and divide it into disjoint subcubes of this radius.  Consider the union of the subcubes $q_j$ which contain a zero. Note that $\hcal^{n-1}(\{u = 0\}\cap Q)$ is roughly the sum of the volumes of the subcubes which have a zero (when the radius is small). Since $\hcal^{n-1}(\{u = 0\}\cap Q)$ is the sum over $A^{n-1}$ such cubes  of  $\hcal^{n-1}(\{u = 0\}\cap q_j) \geq F(N) \frac{1}{A^{n-1}}$,
$$\begin{array}{lll} 2 F(N) \simeq \frac{\hcal^{n-1}(\{u = 0\} \cap Q}{\rm{diam}^{n-1}(Q)} & \geq & \left( (\# {\rm{of\;} \;q_j \; {\rm with \; zeros}) \cdot \;  F(N) \frac{1}{A^{n-1}} } \;  \right) \\&&\\
& \geq & F(N)  A^{n-1} \frac{1}{A^{n-1}} , \end{array} $$
a contradiction if the $N$ is `suffciently large' in the sense of Corollary 6.4. 

Note that these balls are much larger than  wave-length scale balls of radius $\frac{\epsilon}{N}$. It is well-known that the nodal set
is $\frac{1}{\lambda}$-dense in the sense that every wave-length scale ball of radius $\frac{\epsilon}{\lambda}$ contains a zero. However, the Lemma above positions the zero at the centers of the larger balls and gives a lower bound on the number of such disjoint balls. Although used for a different purpose, Donnelly-Fefferman's lower bound also used balls where the nodal
set runs through the center of a ball (deep zeros).  For more on `deep zeros' and further exposition, see  \cite{G16,GM16}.

 By comparison,  the proof in \cite{DF88} is based on the fact  that the doubling index is small in at least half of the balls in a partition by wave-length scale balls. It follows that  $\hcal^{n-1}(B(x_{\nu}, \frac{c}{\lambda}) \geq C \lambda^{-\frac{(n-1)}{2}}$ in half  of the balls (see \cite[p. 164]{DF88}). The remainder of the Donnelly-Fefferman argument sketched above
 does not use real analyticity but is quite different from that of Logunov because Theorem \ref{MLLB} is in terms of the local frequency rather than the global frequency and because this lower bound is not used in \cite{DF88}.

\section{\label{ANALSECT} Sharp upper bounds in the analytic case }

In this section, we sketch the proof of the sharp upper bound in the real
analytic case using analytic continuation to Grauert tubes and  global techiques, following \cite{Zel08,Zel15}.

\begin{theo} \label{NODALBOUND_C14} Let $(M, g)$ be a real analytic Riemannian manifold.  Then, there exists  constants $C, c > 0$ depending only on
 $(M,g)$  so that
$$\hcal^{n-1} (Z_{\phi_{\lambda}}) \leq C \lambda.$$
\end{theo}

  In part, the proof is included   to contrast the techniques in the real analytic case with those in the $C^{\infty}$ case  and in part because the same real analytic techniques are used to determine Hausdorff measures on intersections of nodal sets with real analytic submanifolds in Section \ref{INTERSECT}.  We omit many details that can be found in \cite{Zel08, Zel15} and only highlight the main ideas. The same type of proof also works to given sharp upper bounds for nodal sets of analytic  Steklov eigenfunctions
  \cite{ZSt} (a local approach with non-sharp upper bound is in \cite{BL}.)
  A natural question is whether the proof can be modified to apply to some
classes of non-analytic $C^{\infty}$ metrics using almost analytic extensions to wave-length scale tubes. 
  
  \subsection{Integral geometry}

 Let $N \subset M$ be any smooth hypersurface\footnote{The same formula is true
if $N$ has a singular set $\Sigma$ with $\hcal^{n-2}(\Sigma) < \infty$}, and let $S^*_N M$ denote
the unit covers to $M$ with footpoint on $N$. Then for
$0 < T < L_1,$
\begin{equation} \label{IG} \hcal^{n-1}(N)  = \frac{ 1}{\beta_nT}  
 \int_{S^* M}
\# \{t \in [- T, T]: G^t(x, \omega) \in S^*_N M\} d\mu_L(x,
\omega),\end{equation}
where $\beta_m $ is $2 (m-1)!$ times  the volume of the unit ball in $\R^{m-2}$ and $\mu_L$ is the Liouville measure (here, the Crofton density).
\bigskip

Roughly speaking, the (hyper)-surface measure of $N$ is the average number of intersections with a geodesic arc of length $T$ with $N$.\bigskip

Bounding the number of intersections in the integrand from above (resp.  below)
gives an upper (resp. lower) bound on $\hcal^{n-1}(N)$.

\subsection{Analytic continuation to Grauert tubes}

A real analytic Riemannian manifold $M$
admits a complexification $M_{\C}$, i.e. a complex manifold into
which $M$ embeds as a totally real submanifold. Corresponding to a
real analytic metric $g$ is a unique plurisubharmonic exhaustion
function $\sqrt{\rho}$ on $M_{\C}$ (known as the Grauert tube function) given by  \begin{equation} \label{rhoeq_C14}
\sqrt{\rho}(\zeta) = \frac{1}{2 i} \sqrt{r^2_{\C}(\zeta,
\bar{\zeta})}, \end{equation} where $r^2(x,y)$ is the square of
the distance function and $r^2_{\C}$ is its holomorphic extension
to a small neighborhood of the anti-diagonal $(\zeta,
\bar{\zeta})$ in $M_{\C} \times M_{\C}$.  The open Grauert tube of radius $\tau$
is defined by  $M_{\tau} = \{\zeta \in M_{\C}, \sqrt{\rho}(\zeta)
< \tau\}$.

Since $(M, g)$ is real analytic, the exponential map $\exp_x t \xi$ admits an analytic
continuation in $t$ to imaginary time, and the map
\begin{equation} \label{EXP_C14} E \colon B_{\epsilon}^* M \to M_{\C}, \quad E(x, \xi) =
\exp_x i \xi \end{equation}
is, for small enough $\epsilon$, a
diffeomorphism from the ball bundle $B^*_{\epsilon} M$ of radius
$\epsilon $ in $T^*M$ to the Grauert tube $M_{\epsilon}$ in
$M_{\C}$. We have $E^* (i \ddbar \rho) = \omega_{T^*M}$ and  $E^* \sqrt{\rho} = |\xi|$.  It
follows that $E^*$ conjugates the geodesic flow on $B^*M$ to the
Hamiltonian flow $\exp t \Xi_{\sqrt{\rho}}$ of $\sqrt{\rho}$ with
respect to $\omega$, i.e.
$$E(g^t(x, \xi)) = \exp t \Xi_{\sqrt{\rho}} (\exp_x i \xi). $$

\subsection{Poisson operator and analytic Continuation of eigenfunctions}
The key object in the proof is the  complexified Poisson kernel,
\begin{equation} \label{UI_C14} U(i \tau, \zeta, y) = \sum_{j = 0}^{\infty} e^{-
\tau  \lambda_j} \phi_{ j}^{\C} (\zeta) \phi_j(y), \quad (\zeta,
y) \in M_{\epsilon}  \times M.  \end{equation}  By definition,
\begin{equation} \label{UAC_C14} U_{\C}(i \tau) \phi_j (\zeta) = e^{- \tau \lambda_j} \phi_j^{\C} (\zeta). \end{equation}
The analytic continuability of
the Poisson operator to $M_{\tau}$  implies that  every
eigenfunction analytically continues to the same Grauert tube.

 The following theorem is stated in \cite{Bou_C14}. For proofs, see 
 \cite{Ze12_C14, L}.

\begin{theo}\label{BOUFIO2_C14}    For sufficiently small $\tau > 0$, $  U_{\C} (i \tau): L^2(M)
\to \ocal(\partial M_{\tau})$ is a  Fourier integral
operator of order $- \frac{n-1}{4}$  with complex phase  associated to the canonical
relation
$$\Lambda = \{(y, \eta, \iota_{\tau} (y, \eta) \} \subset T^*M \times \Sigma_{\tau}.$$
Moreover, for any $s$,
$$ U_{\C} (i \tau) \colon W^s(M) \to {\mathcal O}^{s +
\frac{n-1}{4}}(\partial  M_{\tau})$$ is a continuous isomorphism.
\end{theo}

Using the complexified Poisson wave kernel, one can prove the following sup-norm estimate:

\begin{prop} \label{PW_C14} Suppose  $(M, g)$ is real analytic.  Then
$$ \sup_{\zeta \in M_{\tau}} |\phi^{\C}_{\lambda}(\zeta)| \leq C
  \lambda^{\frac{m+1}{2}} e^{\tau \lambda}\quad \text{and} \quad \sup_{\zeta \in M_{\tau}} \bigg|\frac{\partial \phi^{\C}_{\lambda}(\zeta)}{\partial \zeta_j}\bigg| \leq C
  \lambda^{\frac{m+3}{2}} e^{\tau \lambda}.
$$
\end{prop}

\subsection{Complex nodal sets and sequences of logarithms}\label{COMPLEXNODALSECT_C14}

We regard the zero set  $[\ncal_f] $ as a {\it current of integration}, i.e., as a
linear functional on $(n-1, n-1)$ forms $\psi$
$$\langle [\ncal_{\phi_j} ], \psi \rangle = \int_{Z_{\phi_j} } \psi. $$
Recall that a {\it  current} is a linear functional (distribution) on smooth forms. One may use the  K\"ahler hypersurface volume form $\omega_g^{n-1}$ (where $\omega_g = i \ddbar \rho$) to make $\ncal_{\phi_j} $ into a measure:
$$\langle [\zcal_{\phi_j} ], f \rangle = \int_{\zcal_{\phi_j} } f \omega_g^{n-1}, \qquad f \in C(M). $$

The  {\it Poincar\'e-Lelong formula}  gives an exact formula  for the delta-function on the zero set of $\phi_j$ 
\begin{equation} \label{PLLb_C14} \frac{i}{2 \pi}  \ddbar \log  |\phi_j^{\C}(z)|^2 = [\ncal_{\phi_j^{\C}}]. \end{equation}
Thus, if $\psi $ is an $(n-1, n-1)$ form, then
$$ \int_{\ncal_{\phi_j^{\C}}} \psi = \frac{1}{2 \pi} \int_{M_{\epsilon}}  \psi \wedge i \ddbar \log  |\phi_j^{\C}(z)|^2. $$ Existence of such a formula is a key difference between the analytic and $C^{\infty}$ settings.

It follows that to analyse convergence of normalized zero currents it 
suffices to understand convergence of their potentials,
\begin{equation} \label{LOGS}\{ u_j: = \frac{1}{\lambda_{j}} \log |\phi_{j}^{\C}(z)|^2\}_{j = 1}^{\infty}. \end{equation}
A key fact is that this sequence is pre-compact in $L^p(M_{\epsilon})$ for all $p < \infty$
and even that 
 \begin{equation}\bigg\{ \frac{1}{\lambda_{j}} \nabla \log |\phi_{j}^{\C}(z)|^2\bigg\}_{j = 1}^{\infty}. \end{equation}  is pre-compact in $L^1(M_{\epsilon})$.

\subsection{Proof of the Donnelly-Fefferman upper bound}\label{DFUBSECT_C14}

We use the integral geometric formula \eqref{IG} or  ``Crofton formula"\index{Crofton formula} in the  real domain which
bounds  the local   nodal
hypersurface volume above:
\begin{equation} \label{INTGEOM_C14} \hcal^{n-1}(\NPHI \cap U)  \leq C_L \int_{\lcal} \#\{ \NPHI\cap \ell\}
d\mu(\ell). \end{equation}
Here, $\lcal$ is the set of unit line segments.
 In place of line segmens we use geodesic
segments of fixed length $L$,  and parametrize them by $S^*M \times [0, L]$, i.e., by
their initial data and time.  Then
$d\mu_{\ell}$ is essentially Liouville measure $d\mu_L$ on $S^* M$ times $dt$.

The complexification of a
real line $\ell = x + \R v$ with $x, v \in \R^m$ is $\ell_{\C} = x +
\C v$. Since the number of intersection points (or zeros) only increases if we count complex intersections, we have
\begin{equation} \label{INEQ1_C14} \int_{\lcal} \# (\NPHI \cap \ell)\,
d\mu(\ell) \leq \int_{\lcal} \# (\NPHI^{\C} \cap \ell_{\C})\,
d\mu(\ell).
 \end{equation}

Hence to prove Theorem~\ref{NODALBOUND_C14} it suffices to show
\begin{lem} \label{DF2_C14} We have,
$$\hcal^{n-1}(\NPHI) \leq C_L \int_{\lcal} \# (\NPHI)^{\C} \cap \ell_{\C} )\,
d\mu(\ell) \leq C \lambda. $$
\end{lem}

Let $N \subset M$ be a smooth hypersurface in a Riemannian manifold $(M,
g)$.  We denote by  $T^*_N M$ the 
of covectors with footpoint on $N$ and $S^*_N M$ the unit covectors along $N$.
We   introduce  Fermi normal coordinates $(s, y_n) $
along
 $N$,  where  $s$ are coordinates on $N$ and $y_n$  is the
 normal coordinate, so that $y_m = 0$ is a local defining function for $N$.   We also let $\sigma, \xi_m$ be
 the dual symplectic Darboux coordinates. Thus the canonical
 symplectic form is $\omega_{T^* M } = ds \wedge d \sigma + dy_m
 \wedge d \xi_m. $
Let $\pi: T^* M \to M$ be the natural projection. For notational simplicity we denote
$\pi^*y_m$ by  $ y_m$ as functions on $T^* M$. Then $y_m$ is a
defining function of  $T^*_N M$.

  The hypersurface  $S^*_N M \subset S^* M$ is a kind of Poincar\'e section or
symplectic transversal to the  orbits of $G^t$, i.e. is a symplectic transversal away from 
the (at most  codimension one) set  of $(y, \eta) \in S_N^* M$  for which 
$\Xi_{y, \eta} \in T_{y, \eta} S^*_N M$, where as above $\Xi$ is the  generator 
of the geodesic flow. 
\begin{prop}\label{CROFTONEST_C14}  Let $N \subset M$ be any smooth hypersurface\footnote{The same formula is true
if $N$ has a singular set $\Sigma$ with $\hcal^{n-2}(\Sigma) < \infty$}, and let $S^*_N M$ denote
the unit covers to $M$ with footpoint on $N$. Then for
$0 < T < L_1,$
$$\hcal^{n-1}(N)  = \frac{ 1}{\beta_mT}  
 \int_{S^* M}
\# \{t \in [- T, T]: G^t(x, \omega) \in S^*_N M\} \,d\mu_L(x,
\omega),$$
where $\beta_m $ is $2 (m-1)!$ times  the volume of the unit ball in $\R^{m-2}$.
\end{prop}

Put
\begin{equation} \label{acal_C14}
\acal_{L, \epsilon}  \bigg(\frac{1}{\lambda} dd^c \log |\phi_j^{\C}|^2\bigg) = \frac{1}{\lambda} \int_{S^* M} \int_{S_{\epsilon, L}}
dd^c_{t + i \tau} \log |\psi_j^{\C}|^2 (\exp_x (t + i \tau) v)
\,d\mu_L(x, v).
\end{equation}   It is obvious that
\begin{equation} \label{MORE_C14}
\#\{\ncal_{\lambda}^{\C} \cap F_{x,v}(S_{\epsilon, L}) \} \geq \#\{\ncal_{\lambda}^{\R} \cap F_{x,v}(S_{0, L}) \}, 
\end{equation}
since every real zero is a complex zero. 
It follows then from    Proposition~\ref{CROFTONEST_C14} (with $N = \ncal_{\lambda}$) that
$$\acal_{L, \epsilon}  \bigg(\frac{1}{\lambda} dd^c \log
|\phi_j^{\C}|^2\bigg) =   \frac{1}{\lambda} \int_{S^* M}
\#\{\ncal_{\lambda}^{\C} \cap F_{x,v}(S_{\epsilon, L}) \} \,d
\mu(x,v) 
\geq \frac{1}{\lambda} \hcal^{n-1}(\ncal_{\phi_{\lambda}}).$$

Hence to obtain an upper bound on $\frac{1}{\lambda} \hcal^{n-1}(\ncal_{\phi_{\lambda}})$ it suffices
to prove that there exists $M < \infty$ so that
\begin{equation} \label{acalest_C14} \acal_{L, \epsilon}  (\frac{1}{\lambda} dd^c \log
|\phi_j^{\C}|^2)  \leq M. \end{equation}

To prove \eqref{acalest_C14}, we observe that since  $dd^c_{t + i \tau} \log |\phi_j^{\C}|^2 (\exp_x (t + i \tau)
v)$ is a positive $(1,1)$ form on the strip, the integral over
$S_{\epsilon}$ is only increased if we integrate against a
positive smooth test function $\chi_{\epsilon} \in
C_c^{\infty}(\C)$ which equals one on $S_{\epsilon, L}$ and vanishes
off $S_{2 \epsilon, L} $. Integrating  by
parts the $dd^c$ onto $\chi_{\epsilon}$, we have 
\begin{align}
\acal_{L, \epsilon} \bigg(\frac{1}{\lambda} dd^c \log |\phi_j^{\C}|^2\bigg) &\leq   \frac{1}{\lambda} \int_{S^* M} \int_{\C}
dd^c_{t + i \tau} \log |\phi_j^{\C}|^2 (\exp_x (t + i \tau) v)\\
&\qquad \qquad \qquad \times \chi_{\epsilon} (t + i \tau) \,d\mu_L(x, v) \\ 
&=  \frac{1}{\lambda} \int_{S^* M} \int_{\C}
 \log |\phi_j^{\C}|^2 (\exp_x (t + i \tau) v)\\
&\qquad \qquad \qquad \times dd^c_{t + i \tau} \chi_{\epsilon} (t + i \tau)\, d\mu_L(x, v) .
\end{align}

Now write $\log |x| = \log_+ |x| - \log_- |x|$. Here $\log_+ |x| = \max\{0, \log |x|\}$ and
$\log_ |x|= \max\{0, - \log |x| \}. $ Then we need upper bounds for 
$$  \frac{1}{\lambda} \int_{S^* M} \int_{\C}
 \log_{\pm} |\psi_j^{\C}|^2 (\exp_x (t + i \tau) v)
dd^c_{t + i \tau} \chi_{\epsilon} (t + i \tau) \,d\mu_L(x, v) .$$
For $\log_+$ the upper bound is an immediate consequence of Proposition~\ref{PW_C14}.  For $\log_-$ the
bound is subtler: we need to show that $|\phi_{\lambda}(z)| $ cannot be too small on too large a set. 
As we know from Gaussian beams, it is possible that $|\phi_{\lambda}(x) | \leq C e^{- \delta \lambda} $
on sets of almost full  measure in the real domain;
we  need to show that nothing worse can happen. 

 The map \eqref{EXP_C14} is a diffeomorphism and since $B_{\epsilon}^* M = \bigcup_{0 \leq \tau \leq \epsilon} S^*_{\tau} M$
we also have that 
$$E \colon   S_{\epsilon, L} \times S^* M  \to M_{\tau}, \quad E(t + i \tau, x, v) = \exp_x (t + i \tau) v  $$
is  a diffeomorphism for each fixed $t$. Hence  by letting $t$ vary, $E$ 
is a smooth fibration with  fibers given by  geodesic arcs.  Over a point $\zeta \in M_{\tau}$ the fiber of the map is a geodesic arc
$$\{ (t + i \tau, x, v): \exp_x (t + i \tau) v = \zeta, \quad \tau = \sqrt{\rho}(\zeta)\}. $$
Pushing forward the measure $
dd^c_{t + i \tau} \chi_{\epsilon} (t + i \tau) d\mu_L(x, v) $ under $E$ gives  a positive measure $d\mu$ on $M_{\tau}$. 
 A calculation shows that
 it is a smooth multiple $J$  of the K\"ahler volume form $dV_{\omega}$, and we
do not need to know the coefficient function $J$ beyond that it is bounded above and below by constants independent of
$\lambda$.
We then have
\begin{equation} \label{JEN_C14}  \int_{S^* M} \!\int_{\C}
 \log |\phi_j^{\C}|^2 (\exp_x (t + i \tau) v)
dd^c_{t + i \tau} \chi_{\epsilon} (t + i \tau) \,d\mu_L(x, v) = \int_{M_{\tau}} 
\log |\phi_j^{\C}|^2  \, J d V.  \end{equation}
To complete the proof of \eqref{acalest_C14} it suffices  to prove that the right side is $\geq - C \lambda$ for some $ C> 0$.

It follows from a well-known compactness theorem for subharmonic functions that 
  there exists $C > 0$ so that 
	\begin{equation}\label{LOGINT_C14}
	\frac{1}{\lambda} \int_{M_{\tau} } \log |\psi_{\lambda}| \,J d V \geq - C.  \end{equation}
For 
if not, there exists a  subsequence of eigenvalues $\lambda_{j_k}$
so that $\frac{1}{\lambda_{j_k}}\int_{M_{\tau}} \log |\phi_{\lambda_{j_k}}| J d V \to - \infty. $  By Proposition~\ref{PW_C14}, $\{\frac{1}{\lambda_{j_k}} \log |\phi_{\lambda_{j_k}}|\}$   has a uniform upper bound.
Moreover  the sequence does not tend uniformly to $-\infty$  since  $\|\phi_{\lambda}\|_{L^2(M)} = 1$. 
It follows that a further subsequence tends in $L^1$ to a
limit $u$ and by the dominated convergence theorem the limit of \eqref{LOGINT_C14} along the sequence
equals $\int_{M_{\tau}} u \,J dV \not= - \infty$. This contradiction concludes the proof of \eqref{LOGINT_C14}, hence
 \eqref{acalest_C14}, and thus 
 the theorem.

\section{\label{INTERSECT} Intersections of nodal sets with curves and hypersurfaces}

Let $H\subset M$ be a connected, irreducible analytic submanifold. 
Given a submanifold $H \subset M$, we denote the restriction operator to $H$ by $\gamma_H f = f |_H$.

 \begin{defin}\label{DEFINTRO}
Given a subsequence $\scal: = \{\phi_{j_k}\}$, let  \begin{equation} \label{ujdef} u_j: =  \frac{1}{\lambda_j} \log |\phi_j |^2 \end{equation} and denote their restrictions to $H$ by
restrictions  \begin{equation} \label{ujCdef} \gamma_{H} u_j: =  \frac{1}{\lambda_j} \log |\phi_j^{H}|^2 \end{equation}
to $H$.  We say
that a connected, irreducible real analytic submanifold $H \subset M$ is {\it $\scal$-good}, or that
$(H, \scal)$ is a good pair,  if the sequence  \eqref{ujCdef} with $j_k \in \scal$ does
{\bf not} tend to $-\infty$ uniformly on compact subsets of  $H$,
 i.e.  there exists
a constant $M_{\scal} > 0$ so that  $$ \;\;\;\sup_H u_j^{H} \geq - M_{\scal}, \;\; \forall j \in \scal 
.$$ 
 If $H$ is $\scal$-good when $\scal$ is the entire   orthonormal basis sequence, we say that $H$ is {\it completely good}.
   \end{defin}
 The connected, irreducible assumption is made to prohibit taking
unions $H_1 \cup H_2$ of two analytic submanifolds, one of which may be good and the other bad. By the definition above, the union would be good even though one component is bad.

The most extreme example of a bad pair consists of a submanifold $H$ on
which a sequence $\scal$ of eigenfunctions vanishes. The problem of characterizing such ``nodal hypersurfaces'' was posed by Bourgain-Rudnick \cite{BR11,BR12} and studied by them on flat tori. At this time, every known bad pair is  nodal.

 As above, we denote the nodal set of an eigenfunction $\phi_{\lambda}$
of eigenvalue $- \lambda^2$ by
$$\ncal_{\phi_{\lambda}}= \{x \in M:
\phi_{\lambda} (x) = 0\}. $$  In \cite{TZ17} is proved the folllowing:

\begin{theo} \label{INTER} Suppose that $(M^m,g)$ is a real analytic Riemannian manifold
of dimension m without boundary. Let  $\ccal \subset M$, resp. $H \subset M$, be a connected, irreducible real analytic
curve (resp. hypersurface). If $\ccal$ (resp. $H$) is $\scal$-good, then there exists a constant $A_{\scal, g}$ so that for $(j_k \in \scal)$,
$$\left\{ \begin{array}{ll} n(\phi_{j_k}, \ccal) : = \# \{\ccal \cap \ncal_{\phi_j}\} \leq A_{\scal, g} \; \lambda_{j_k}, \;\; , & \dim \ccal = 1), \\ & \\
\hcal^{n-2} (\ncal_{\phi_{j_k}}  \cap H)  \leq A_{\scal, g} \; \lambda_{j_k}, \;\;, & (\dim H = n-1). \end{array} \right.$$
\end{theo}

As in Section \ref{ANALSECT} the upper bound is proved by analytic
continuation of the eigenfunctions and curves to the complexification of $M$.
 Complexification is useful for upper bounds since the number $n(\phi_{\lambda}^{\C}, \ccal_{\C})$  of zeros of the complexified
eigenfunction on the complexified curve is $\geq$ the number of real zeros, 
i.e.
\begin{equation} \label{MORE} n(\phi_{\lambda}^{\C}, \ccal_{\C}): = 
\#\{\ncal_{\lambda}^{\C} \cap \ccal_{\C} \} \geq n(\phi_{\lambda}, \ccal):= \#\{\ncal_{\lambda}^{\R} \cap \ccal \}.
\end{equation}



\subsection{\label{MICROGOOD} Dynamical conditions for goodness} The `goodness'  hypothesis in Theorem \ref{INTER} obviously needs to be explored. The next result gives a dynamical condition for  almost complete goodness of a hypersurface in the strong sense that the restrictions possess uniform lower bounds in the sense just mentioned. The criterion consists of two conditions on $H$: (i) asymmetry 
with respect to geodesic flow, and (ii) a full measure flowout condition.

  We begin with (i).    In \cite{TZ13}, a geodesic asymmetry condition on a hypersurface  was introduced
     which is sufficient that restrictions of quantum ergodic eigenfunctions on $M$  remain quantum ergodic on the hypersurface.  
  It turns out that
the same asymmetry condition plus a flow-out condition implies that   a
     hypersurface is good for a density one  subsequence  of eigenfunctions  and that for any $\delta> 0$, the $L^2$ norms of the restricted eigenfunctions
     have a uniform lower bound $C_{\delta} > 0$ for a subsequence of density $1 - \delta$.   The asymmetry condition pertains to the
 two `sides' of $H$, i.e. to the two lifts of $(y, \eta) \in B^* H$ to unit 
 covectors $\xi_{\pm}(y, \eta) \in S^*_H M$ to $M$. We denote the symplectic volume measure on $B^* H$ by $\mu_H$. We define the symmetric
 subset $B_S^* H$ to be the set of $(y, \eta) \in B^*H$ so that 
 $r_H G^t(\xi_+(y, \eta)) = G^t(\xi_-(y, \eta))$ for some $t \not= 0$. Here, $r_H$ is reflection through $TH$.
     
     \begin{defin} \label{MADEF} $H$ is microlocally asymmetric if $\mu_H(B_S^*H) = 0$. \end{defin}

                 Next we turn to the  flow-out condition (ii). It  is that \begin{equation} \label{ASSUME}
\mu_L ( \rm{FL}(H) ) = 1, \;\; \rm{where}\;
 {\rm{FL}(H)}: =   \bigcup_{ t \in \R} G^t ( S_H^*M \setminus S^*H) \;\;  \end{equation}   
      is the geodesic  flowout of 
of the non-tangential unit cotangent vectors $S^*_H M \setminus S^*H$ along $H$.  In other words, almost all geodesics intersect $H$.   In \cite{TZ17} it is shown that a large class of curves satisfy (\ref{ASSUME}) on  surfaces with completely  integrable geodesic flows, including convex surfaces of revolution and Liouville tori satisfying generic twist assumptions. Ergodicity is thus not assumed.

\begin{theo} \label{MASSMICRO} Suppose that $H$ is a  microlocally asymmetric hypersurface satisfying \eqref{ASSUME}.

Then: if  $\scal = \{\phi_{j_k}\}$ is a sequence of eigenfunctions
satisfying $ ||\phi_{j_k} |_H||_{L^2(H)} = o(1)$, then the upper density
$D^*(\scal)$  equals zero. \end{theo}

 The following theorem gives a more quantitative version:

\begin{theo} \label{mainthm1}
Let $H \subset M$ be a microlocally asymmetric hypersurface satisfying \eqref{ASSUME}. 
Then, for any $\delta >0,$ there exists a subset $\scal(\delta) \subset \{1,...,\lambda \}$ of density $D^*(\scal(\delta)) \geq 1-\delta$ such that 
$$ \| \phi_{\lambda_j} \|_{L^2(H)} \geq C(\delta) >0,\quad j \in \scal(\delta).$$
\end{theo}




As mentioned above, the  assumption $ ||\phi_{j_k} |_H||_{L^2(H)} = o( 1)$ is  much 
weaker than the $\scal$- badness of $H$.
In fact, we do not know any microlocal (or other techniques) that prove
goodness without proving the stronger positive lower bound. There
do exist other non-microlocal  techniques which directly prove goodness.  In
 \cite{JJ14}, J.  Jung
     proved that geodesic distance circles and horocycles in the hyperbolic plane are good relative to eigenfunctions on compact or finite area hyperbolic surfaces.

A combination of Theorems \ref{INTER} and \ref{mainthm1} gives 

\begin{cor} The nodal intersection upper bounds of Theorem \ref{INTER} are valid for asymmetric hypersurfaces satisfying \eqref{ASSUME}. \end{cor}


 
%

\subsection{\label{RELATE} Relating weak* limits on $M$ and on $H$}

The main step in proving Theorems \ref{MASSMICRO}-
 \ref{mainthm1} is to relate weak* limits or microlocal defect measures on $M$ and on $H$.
 We recall that an invariant measure $d\mu$ for the geodesic flow on $S^*M$ is
called a microlocal defect (or defect measure, or quantum limit) if there exists a sequence $\{\phi_{j_k}\}$ of eigenfunctions
such that $\langle A \phi_{j_k}, \phi_{j_k} \rangle_{L^2(M)} \to \int_{S^*M} \sigma_A d\mu$ for all pseudo-differential operators
$A \in \Psi^0(M)$. There are analogous notions for semi-classical pseudo-differential operators. We  refer to \cite{Zw} for background.

 There is an obvious relation between matrix elements on $M$ and matrix elements on $H$.  It involves a time average $\VT(a)$ of $\gamma_H^* Op_h(a) \gamma_H$.
In \cite{TZ13}, $\VT(a)$ was decomposed into a pseudo-differential term $P_{T, \epsilon}$ and a Fourier integral term $F_{T, \epsilon}$.  The symbol of  $P_{T, \epsilon}$ is essentially
a flow-out of $a$ using that $S^*_H M$ is a sort-of cross-section to the geodesic flow

\begin{prop} \label{FTPROP} Suppose that $H$ is asymmetric. Then, for any $T,\epsilon >0$ there exists a density-one sequence $\scal_F(T,\epsilon)$ such that   for $a \in S^0(H),$ $$\lim_{k \to \infty; \, j_k \in \scal_F(T,\epsilon)}  \Big( \langle  \langle (1-\chi_{\epsilon}^H(h_{j_k}) )  \,  Op_H(a) \, \phi_{j_k} |_H, \phi_{j_k} |_H \rangle_{L^2(H)}  - \langle P_{T, \epsilon}(a) \phi_{j_k}, \phi_{j_k} \rangle_{L^2(M)} \Big) = 0. $$
\end{prop} 
Here, $\chi_{\epsilon}^H(h_{j_k})$ denotes  a semiclassical pseudodifferential quantization of the cutoff $\chi_{\epsilon}^H \in C^{\infty}_0(T^*H)$ with $h_{j_k} = \lambda_{j_k}^{-1}.$

\subsection{Observability and Control estimates}

Theorems  \ref{MASSMICRO} and \ref{mainthm1} are referred to as `geometric control' estimates. In this section, we briefly review some classic and recent results of this kind to situate the results in a broader context.
They also give rise to a question about improving the theorems.

Observability and control estimates pertain to  open sets $U$
where one can obtain lower bounds on local $L^2$ norms of eigenfunctions or solutions of related wave equations.  
Let $\omega \subset M$ be an open set. The geometric control condition is that 
$$\mbox{Any geodesic meets }\; \omega\; \mbox{in a time }\; t \leq T_0. $$

  An influential article on the problem is \cite{BLR}.

In \cite{LR95}, Lebeau-Robbiano proved two relevant theorems of this type
regarding the spectral projections kernels $\Pi_{[0, \lambda]}$ of a Laplacian.
Both give lower $L^2$ bounds for linear combinations of eigenfunctions of eigenvalue $\leq \lambda^2$ on a small ball:

\begin{theo} Let $\Omega$ be a $C^2$ domain and consider the Dirichlet
eigenfunctions and specral projections. For each $0 < R \leq 1$ there exists $N(\Omega, R)$ such that
if $B_{4 R}(x_0) \subset \Omega, f \in L^2(\Omega)$ then
$$||\Pi_{[0, \lambda]} f||_{L^2(\Omega)} \leq N e^{N \sqrt{\lambda}} 
||\Pi_{[0, \lambda]} f||_{L^2(B_R(x_0))}.$$
\end{theo}

\begin{theo} Let $\omega \subset \Omega$ be an open subset. Then  
$$\int_{\omega} \left| \sum_{j: \lambda_j \leq \lambda} a_j \phi_j(x)\right|^2 dV  \geq C e^{- c \sqrt{\lambda}} \sum_j |a_j|^2. $$
\end{theo}

A recent very nice geometric control relaxes the geometric control condition
while obtaining similar lower bounds.   Theorem 2.5 of \cite{AR}: 

\begin{theo} Let $(M, g)$ be a compact Riemannian manifold of dimension d and constant curvature $-1$. 
Let $a \in C^{\infty}(M)$ and define the $G^t$- invariant subset of $S^*M$ 
$$K_a = \{(\rho \in S^*M: a^2(G^t(\rho)) = 0, \;\;\; \forall t \in \R\}. $$
Assume that the topological entropy of $K_a$ is $\leq \frac{d-1}{2}$. Then for all $T > 0$ there
exists $C_{T,a} > 0$ so that for all $u \in L^2(M)$,
$$||u||_{L^2}^2 \leq C_{T,a} \int_0^T ||a e^{i t \Delta/2} u||_{L^2} dt. $$

\end{theo}

The condition is satisfied if the Hausdorff dimension of $K_a$ is $\leq d$. Example: Let $\gamma$ be a closed
geodesic and a small tubular neighborhood of $\gamma$ that does not contain another complete geodesic. Let
$a > 0$ in the complement of this tubular neighborhood and $z = 0$ near $\gamma$. Then $K_a = \gamma$. 

A recent breakthrough result of Dyatlov-Jin \cite{DJ17} is the following
\begin{theo} \label{DJ} Let $(M,g)$ be a compact hyperbolilc surface. Let $a \in C_0^{\infty}(T^* M)$ with  $a |_{S^*M} $ not identically zero. Let $u$ be
an eigenfunction of eigenvalue $\lambda^2$ and $||u||_{L^2} = 1$. Then
there exists a constant $C_a$ independent of $\lambda$ so that
$$||Op_h(a) u||_{L^2} \geq C_a. $$

\end{theo}
Here, $Op_h(a)$ is the semi-classical pseudo-differential operator with
symbol $a$. If $a(x, \xi)  = V(x)$ is a multiplication operator, one gets that
$\int_B |u|^2 dV \geq C_B > 0$, i.e. a uniform lower bound of the $L^2$ mass on all balls. A corollary of Theorem \ref{DJ} is that all quantum limits of sequences of eigenfunctions on compact hyperbolic surfaces have full support in $S^* M$, i.e. charge every open set.

By comparison, Theorems \ref{MASSMICRO} and \ref{mainthm1} give 
lower bounds for $L^2$ mass on a hypersurface $H$ for a subsequence of density one. The possible improvement alluded to above is whether these results can be  combined with Theorem \ref{DJ} on a compact hyperbolic surface to prove 
\begin{conj} 
Let $H \subset M$ be an asymmetric  curve of a compact hyperbolic surface  satisfying \eqref{ASSUME}. 
Then there exists  $C >0$ such that 
$ \| \phi_{\lambda_j} \|_{L^2(H)} \geq C. $
\end{conj} The geometric control condition automatically holds for any curve of a compact hyperbolic surface. Asymmetry cannot be dropped as a condition in view of the odd eigenfunctions of a hyperbolic surface with involution. Although curves and open sets seem quite different, the unit tangent bundle $S^*_H M$ along the curve provides a cross-section to the flow and its flow-out contains an open set. In particular, any quantum limit (microlocal defect measure) of a sequence of eigenfunctions which has a `hole' in its support would also have a hole in the cross-section.   The idea is that  Theorem \ref{DJ} generalizes to give lower bounds $||Op_h(a) \phi_j |_H||_{L^2(H)}$ under the same hypotheses together with the asymmetry condition on $H$. One of the main obstructions to proving this is that a
zero density sequence was  thrown out in Theorem \ref{MASSMICRO} comes because of the Fourier integral operator  $F_{T, \epsilon}$ in Section \ref{RELATE}. To generalize Theorem \ref{DJ} one would need to prove
that $\langle F_{T, \epsilon} \phi_j, \phi_j \rangle \to 0$ for the full sequence.
It is possible that the FUP (fractal uncertainty principle) of \cite{DJ17} would exclude this in somewhat the way it excluded holes in the support of the pseudo-differential term. Namely if $\langle F_{T, \epsilon} \phi_j, \phi_j \rangle $ does not tend to zero, it induces a microlocal defect measure on the canonical relation of $F$ which is supported in a very thin subset, possibly of the type exluded by the FUP.

\section{Lower bounds on numbers of nodal domains}

Let $(M, g)$ be a compact negatively curved surface without boundary, and consider an
eigenfunction of the Laplacian \eqref{EP}. This section is concerned with the nodal domains, 
$$M \backslash \ncal_{\phi_{\lambda}}= \bigcup_{j = 1}^{N(\phi)}  \Omega_j. $$ 
The Courant upper bound states that $N(\phi_j) \leq j$. H. Lewy showed that there is no non-trivial lower bound by constructing infinite sequences of spherical harmonics of growing degree with only two or three nodal domains.
The question arises whether a (possibly generic) $(M,g)$ possesses any sequence of eigenfunctions for which $N(\phi_{j_k}) \to \infty$. It seems like the answer should be `yes' but so far the problem is open. At this time, the only infinite dimensional  class of Riemanniam manifolds or billiard tables which are known to possess sequences of eigenfunctions for which $N(\phi_{j_k}) \to \infty$ are certain ergodic ones. This includes  certain negatively curved surfaces and non-positively curved surfaces with concave boundary. In both cases, the mechanism producing many nodal domains is a distinguished curve playing the role of a boundary. 
An example
of a non-positively curved surface with concave boundary is a Sinai-Lorentz 
billiard in which one removes a small disc $D$ from  $X$.

\begin{theo}\label{JJtheo1}[J.Jung-Z, 2014]
Let $(X, g)$ be a surface with curvature $k \leq 0$ and with concave boundary.  Then for any orthonormal eigenbasis $\{\phi_j\}$ of Dirichlet (or Neumann) eigenfunctions, one can find a density $1$ subset $A$ of $\mathbb{N}$ such that
\[
\lim_{\substack{j \to \infty \\ j \in A}}N(\phi_j) = \infty,
\]

\end{theo}

A density one subset $A \subset {\bf N}$ is one for which
$\frac{1}{N} \#\{j \in A, j \leq N \} \to 1, \;\; N \to \infty. $ 
The first result on counting nodal domains by counting intersections with a curve was proved by Ghosh-Reznikov-Sarnak
for $M = {\mathbb H}^2/SL(2, \Z)$. Theorem \ref{JJtheo1} has been extended to  a rather general class of billiard tables with ergodic billiard flow by H. Hezari in \cite{He16b}.

A closely related result pertains to negatively curved surfaces possessing
an isometric involution with non-empty fixed point set. In this case, the fixed point set is a finite union of closed geodesics, playing the role of a boundary.

\begin{theo}\label{theo1}(J. Jung, S. Z. (2013-4))
Let $(M,J,\sigma)$ be a compact real Riemann surface with $\mbox{Fix}(\sigma) \neq \emptyset$ and
dividing.\footnote{For odd eigenfunctions, the  conclusion holds as long  as $Fix(\sigma)\neq \emptyset$.
Later the `dividing' assumption was removed.} Let $g$ be any $\sigma$-invariant Riemannian metric. 
  Then for any orthonormal eigenbasis $\{\phi_j\}$ of $L_{even}^2(Y)$, resp. $\{\psi_j\}$ of $L_{odd}^2(M)$, one can find a density $1$ subset $A$ of $\mathbb{N}$ such that
\[
\lim_{\substack{j \to \infty \\ j \in A}}N(\phi_j) = \infty,
\]
resp.
\[
\lim_{\substack{j \to \infty \\ j \in A}}N(\psi_j) = \infty,
\]

\end{theo}



Above, we  assume $M$  is a Riemann surface of genus $\mathfrak{g}$   (with complex structure $J$)
possessing an anti-holomorphic
involution $\sigma$   whose fixed point set $\mathrm{Fix}(\sigma)$  is non-empty.
 Define
$\mathcal{M}_{M, J, \sigma}$ to be the space  of  $C^{\infty}$  $\sigma$-invariant {\it negatively curved} Riemannian metrics on a real Riemann surface
$(M, J, \sigma)$.
$\mathcal{M}_{M, J, \sigma}$ is an open set in the space of $\sigma$-invariant metrics, and in particular is infinite dimensional.
For each $g \in {\mathcal M}_{M, J, \sigma}$,  the fixed point set $\mathrm{Fix}(\sigma)$
is a disjoint union
\begin{equation} \label{H}   \mathrm{Fix}(\sigma)  = \gamma_1 \cup \cdots \cup \gamma_n \end{equation} of $0 \leq n \leq \mathfrak{g} + 1$ simple closed geodesics. 

Remark: J.J. and S. Jang have since  proved stronger results of this kind.

The isometry $\sigma$ acts  on $L^2(M, dA_g)$, and we define
$L^2_{even}(M)$, resp. $L^2_{odd}(M)$,  to denote the subspace of even
functions $f(\sigma x) = f(x)$, resp. odd elements $f(\sigma x) = - f(x)$.\bigskip

 Even and odd parts of eigenfunctions  are eigenfunctions, and all eigenfunctions
are linear combinations of   even or odd eigenfunctions. We denote by $\{\phi_j\}$  an orthonormal basis of   $L_{even}^2(M)$ of  even eigenfunctions, resp. $\{\psi_j\}$ an orthonormal basis   of $L_{odd}^2(M)$ of odd eigenfunctions. \bigskip

   .For
generic metrics in $\mathcal{M}_{M,J, \sigma}$, the eigenvalues are simple (multiplicity one) and therefore all eigenfunctions are either
even or odd.

Next is a quantitative  lower bound in the second case.

\begin{theo}\label{theo1b}(S.Z. 2015-16)
Let $(M,J, \sigma)$ be a  compact real Riemann surface of genus $\mathfrak{g} \geq 2$ with anti-holomorphic involution $\sigma$ satisfying  $Fix(\sigma)\neq \emptyset$. Let $\mathcal{M}_{M, J, \sigma}$
be the space of $\sigma$-invariant negatively curved $C^\infty$ Riemannian metrics on $M$.Then for
any $g \in \mathcal{M}_{(M, J, \sigma)}$ and  any orthonormal $\Delta_g$-eigenbasis $\{\phi_j\}$ of $L_{even}^2(M)$, resp. $\{\psi_j\}$ of $L_{odd}^2(M)$, one can find a density $1$ subset $A$ of $\mathbb{N}$ and a constant $C_g > 0$ depending only on $g$ such that, for $ j \in A$
\[N(\phi_j)  \geq 
 C_g\; (\log \lambda_j)^{K}, \;\;  (\forall K < \frac{1}{6}).
\]
resp.
\[
N(\psi_j) \geq C_g \;
 (\log \lambda_j)^{K} , \;\; (\forall K < \frac{1}{6}).
\]


\end{theo}

\subsection{Sketch of the proofs}

\begin{enumerate}

\item Show that the number $N(\phi_{\lambda})$ of nodal domains is $\geq  \half
N(\phi_{\lambda} |_{\partial M}, 0, \partial M)$, the number of zeros of $\phi_{\lambda}$ on
the boundary (in the Neumann case).
 {\bf This is purely topological and is why we need $\partial M \not= \emptyset$.}
\bigskip

\item {\bf Prove that  Neumann eigenfunctions have a lot of zeros on $\partial M$,
resp. Dirichlet eigenfunctions have many zeros of $\partial_{\nu} \phi_j = 0$
on $\partial M$. This is where ergodicity is used: Neumann eigenfunctions, restricted to the boundary, are ``ergodic''. \footnote{ the quantum ergodic restriction
theorem of Hassel-Z and of Christianson-Toth-Z.}
}
\bigskip

\item To prove (2), we show that $\int_{\beta} \phi_j ds << \int_{\beta}| \phi_j |ds$
on any arc $\beta \subset \partial M$. 


\item To get a log lower bound, prove this for
$|\beta| \leq (\log \lambda)^{- 1}$. Use  that the recent log-scale quantum ergodicity results of
Hezari-Riviere and X. Han restrict  to curves on surfaces.

\end{enumerate}

\begin{theo} \label{useful}[Christianson-Toth-Z, 2013] Let $\gamma$ be either $\partial M$ for the surface with boundary or
$Fix(\sigma)$ for the surface with involution.  Then, for a subsequence of
Neumann eigenfunctions
of density one,

$$\begin{array}{l}
 \int_{\gamma} f  \phi_{j}^2 ds  \rightarrow \frac{4}{ 2 \pi \mbox{Area}(M)} \int_{\gamma} f(s)   d s.
\end{array}$$

Similarly for normal derivatives of Dirichlet eigenfunctions.  Cauchy data of eigenfunctions to $\gamma$ are quantum ergodic along $\gamma$.
This is part of a much more general result.

\end{theo}

If  \begin{enumerate}

\item  $\int_{\gamma} f \phi_{\lambda_j} ds = O(\lambda_j^{-\half} (\log \lambda_j)^{1/4})$,
\item  $\int_{\gamma} f \phi_{\lambda_j}^2 ds \geq 1$,
\item   $ ||\phi_j||_{L^{\infty}} \leq C \frac{\lambda_j^{\half}}{\sqrt{\log \lambda_j}}$, 
\end{enumerate} there must exist an  unbounded number of sign changes of $\phi_{\lambda_j} |_{\gamma}$ as $j \to \infty$. \bigskip Indeed, for
any arc $\beta \subset \gamma$,
$$|\int_{\beta} \phi_{\lambda_j} ds| \leq C \lambda_j^{-1/2} (\log \lambda)^{1/4} $$
and
$$\int_{\beta} |\phi_{\lambda_j} | ds \geq ||\phi_{\lambda_j} ||_{\infty}^{-1} ||\phi_{\lambda_j} ||_{L^2(\beta)}^2
\geq C \lambda_j^{-1/2} (\log \lambda_j)^{\half}, $$
and this is a contradiction if $\phi_{\lambda_j}  \geq 0$ on $\beta $.

If  \begin{enumerate}

\item  $\int_{\gamma} f \phi_{\lambda_j} ds = O(\lambda_j^{-\half} (\log \lambda_j)^{1/4})$ (Density one subsequence; Kuznecov); 
\item  $\int_{\gamma} f \phi_{\lambda_j}^2 ds \geq 1$ (Den 1 subseq; QER = quantum ergodic restriction);
\item   $ ||\phi_j||_{L^{\infty}} \leq C \frac{\lambda_j^{\half}}{\sqrt{\log \lambda_j}}$ (log improvement on canonical sup norm bound);
\end{enumerate} there must exist an  unbounded number of sign changes of $\phi_{\lambda_j} |_{\gamma}$ as $j \to \infty$. \bigskip Indeed, for
any arc $\beta \subset \gamma$,
$$|\int_{\beta} \phi_{\lambda_j} ds| \leq C \lambda_j^{-1/2} (\log \lambda)^{1/4} $$
and
$$\int_{\beta} |\phi_{\lambda_j} | ds \geq ||\phi_{\lambda_j} ||_{\infty}^{-1} ||\phi_{\lambda_j} ||_{L^2(\beta)}^2
\geq C \lambda_j^{-1/2} (\log \lambda_j)^{\half}, $$
and this is a contradiction if $\phi_{\lambda_j}  \geq 0$ on $\beta $.  


These estimates are more difficult in the boundary case. The boundary estimates  are discussed in Section \ref{RSECT}.

\subsection{Log scale QER theorems in negative curvature} To prove Theorem \ref{theo1b},  we need to improve the estimates to show that
a full density quantum ergodic sequence has a sign-changing zero on
logarithmically shrinking arcs of the axis of symmetry. The length scale is
  \begin{equation} \label{ell} \ell_j =  |\log \hbar|^{-K} = (\log \lambda_j)^{- K}  \;\rm{ where}\; 0 < K < \frac{1}{3d}. \end{equation}
We partition $\rm{Fix}(\sigma)$ into $\ell_j^{-1}$  open intervals of lengths $\ell_j$  and show that $u_j$
has a sign changing zero in each interval. \bigskip
We 
choose a cover of $ \rm{Fix}(\sigma)$ by  $C \ell^{-1}$  balls   of radius $\ell$ with centers $\{x_k\} \subset  $ at a net of points
of $\rm{Fix}(\sigma)$ so that  
\begin{equation} \label{COVER} \rm{Fix}(\sigma) \subset \bigcup_{k = 1}^{R(\ell)} B(x_k, C \ell ) \cap \rm{Fix}(\sigma). \end{equation}
\bigskip


\begin{itemize}

\item (i)\; One needs to prove a QER  (quantum ergodic restriction) theorem  on the length
scale $O(\ell_j)$, which says (roughly speaking) that there exists a subsequence of eigenfunctions $u_{j_n}$ of density one so that 
matrix elements of the restricted eigenfunctions tend to their Liouville limits simultaneously for all balls of the cover.
Since there are $(\log \lambda)^K$ such balls, the scale $(\log \lambda)^{-K}$ of the  QER theorem is constrained. \bigskip

\item (ii)\; One needs to prove that there exists a
subsquence of density one for which $\int_{\beta_{n}} u_{j_k}$ is of order  $|\beta_n|\lambda_j^{-\frac{1}{4}} (\log \lambda_j)^{1/3}$  simultaneously for all the balls
$\beta_{n}$ of the cover. The Kuznecov estimates are

\begin{prop} \label{LOGKUZ} Let $K$ be as in \eqref{ell} and $\{x_k\}$ the centers of \eqref{COVER}.  Then for a subsequence of $\Lambda_K \subset {\mathbb N}$ of density one, if $j_n \in \Lambda_K$,

$$ \left| \int_{B(x_k, C \ell) \cap H}  \phi_{j_n} ds \right|  \leq C_0  \ell_j \; \lambda_j^{-\frac{1}{4}} \;(\log \lambda_j)^{1/3}= C_0 (\log \lambda_j)^{- K} \lambda_j^{-\frac{1}{4}} \;(\log \lambda_j)^{1/3} ,$$ 
resp.
$$ \left| \int_{B(x_k, C \ell) \cap H}  \lambda_{j_k}^{-\frac{1}{2}} \partial_{\nu} \psi_{j_n} ds \right|  \leq C_0  \ell_j \; \lambda_j^{-\frac{1}{4}} \;(\log \lambda_j)^{1/3} = C_0 (\log \lambda_j)^{- K} \lambda_j^{-\frac{1}{4}} \;(\log \lambda_j)^{1/3} ,$$ 
uniformly in $k$.
\end{prop}

\bigskip

\item (iii) \;  The sup-norm estimate $||u_j||_{\infty} = O(\frac{\lambda_j^{\frac{1}{4}}}{\sqrt{\log \lambda_j}})$ does not need to be modified.

\end{itemize}
\bigskip




These estimates imply a kind of uniform log-scale quantum ergodicity:

\begin{cor} \label{BIGintro} Let $(M, J, \sigma, g)$ be a negatively curved surface with isometric
involution. Then  for any orthonormal basis of even eigenfunctions $\{\phi_j\}$, resp.
odd eigenunctions $\{\psi_j\}$,  there exists a full density subsequence $\Lambda_K$  so that for $j_n \in \Lambda_K$,  

$$\int_{B(x_k, C \ell ) \cap H}  |\phi_{j_n}|^2   dS_g   \geq  a_1 (\log \lambda_{j_n})^{-  K}$$
and
$$\int_{B(x_k, C \ell ) \cap H}  |\lambda_j^{-\half}  \partial_{\nu} \psi_{j_n}|^2 dS_g   \geq a_1  (\log \lambda_{j_n})^{-  K}$$
uniformly in $k$.

\end{cor}
Substituting  the above estimates into the proof of Theorem \ref{theo1} proves Theorem \ref{theo1b}.

\subsection{A word on the topological argument}

For the sake of completeness, let us  sketch the topological argument.
Let $\gamma = \partial M, $ resp. $\mbox{Fix}(\sigma)$.

One can modify the nodal set $\ncal_{\phi_j}$  ($\ncal_{\phi_j} \cup \gamma$, when $\phi_j$ is even)   to give it the  structure of an embedded graph:
\begin{enumerate}
\item For each embeded circle which does not intersect $\gamma$, we add a vertex.
\item Each singular point $\phi_j (p) = d \phi_j(p) = 0$ is a vertex.
\item If $\gamma \not\subset \ncal_{\phi_\lambda}$, then each intersection point in $\gamma \cap \ncal_{\phi_j}$ is a vertex.
\item Edges are the arcs of $\ncal_{\phi_j}$ ($\ncal_{\phi_j} \cup \gamma$, when $\phi_j$ is even) which join the vertices listed above.
\end{enumerate}

This way, we obtain a graph  embeded into the surface $M$. 
 An  embedded graph $G$ in a surface
$M$ is a finite set $V(G)$ of vertices and a finite set $E(G)$ of edges which are simple (non-self-intersecting)
curves in $M$ such that any two distinct edges have at most one endpoint and no interior points in common.
The {\it faces} $f$ of $G$ are the  connected components of $M \backslash V(G) \cup \bigcup_{e \in E(G)} e$.
The set of faces is denoted $F(G)$. An edge $e \in E(G)$ is {\it incident} to $f$ if the boundary of $f$ contains
an interior point of $e$. Every edge is incident to at least one and to at most two faces; if $e$ is incident
to $f$ then $e \subset \partial f$.
The faces are not assumed to be cells and the sets $V(G), E(G), F(G)$ are
not assumed to form a CW complex.

Now let $v(\phi_\lambda)$ be the number of vertices, $e(\phi_\lambda)$ be the number of edges, $f(\phi_\lambda)$ be the number of faces, and $m(\phi_\lambda)$ be the number of connected components of the graph. Then by Euler's formula
\begin{equation}\label{euler}
v(\phi_\lambda)-e(\phi_\lambda)+f(\phi_\lambda)-m(\phi_\lambda) \geq 1- 2 g_M
\end{equation}
where $g_M$ is the genus of the surface.\bigskip

The  Euler inequality  gives a lower bound for the number of nodal domains 
from a lower bound  on the number of
points where $\partial_{\nu} \phi_j = 0$ and changes sign on  $\gamma$ (odd case, Dirichlet), resp. numbers of sign-change zeros on $\gamma$ (even, Neumann case). 

\begin{lem}\label{lem1}
For an odd eigenfunction $\psi_j$, let $\Sigma_{\psi_j} = \{x: \psi_j(x) =
d \psi_j(x) = 0\}$. Then
\[
N(\psi_j) \geq \#\left(\Sigma_{\psi_j}\cap \gamma\right) +2 - 2g_M,
\]
and for an even eigenfunction $\phi_j$,
\[
N(\phi_j) \geq \frac{1}{2}\#\left(\ncal_{\phi_j} \cap \gamma\right)+1-g_M.
\]
\end{lem}

\begin{center}
\includegraphics[scale=0.8]{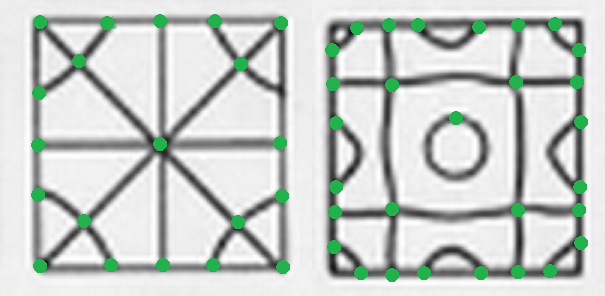}
\end{center}

\subsubsection{Extension to non-positively curved surfaces with concave boundary}

At the present time, the analogous logarithmic improvement of Theorem \ref{JJtheo1}  is lacking two  ingredients. First is the analogue of Proposition \ref{LOGKUZ} and the sup-norm estimate
$  ||\phi_j |_{\partial M}||_{L^{\infty}} \leq C \frac{\lambda_j^{\half}}{\sqrt{\log \lambda_j}}$ (in the Neumann case, with the normal derivative modification in the Dirichlet case).

\subsubsection{Infinite area hyperbolic surfaces}

In \cite{JN17}, Jakobson-Naud use a related argument to obtain $\geq C \lambda$ zeros on an axis of symmetry and therefore $C \lambda$ nodal domains for Eisenstein series on an infinite area convex co-compact hyperbolic surface when $\delta(\Gamma) < \half$. For purposes of this paper, the condition means that the Eisenstein  series converges and is dominated by its first term. The improvement on the number of zeros on the axis and the number of nodal domains is due to the uniform boundedness of the local $L^{\infty}$-norm of the Eisenstein series, which replaces the bound
(3) below Theorem \ref{useful}.

\bibliography{jdg}
\bibliographystyle{plain}

\section{\label{RSECT} Eigenfunction restriction theorems}

Theorem \ref{JJtheo1} required two  types of eigenfunction restriction theorems: (i) an estimate of the `period'  $\int_H f \phi_j dS$ of an eigenfunction over a curve (or  hypersurface) and (ii) a sup norm estimate of the Cauchy
data of the eigenfunction on the distinguished curve, essentially a boundary. 
Both topics belong to a stream of results on integrals of eigenfunctions over submanifolds or $L^p$ norms of eigenfunctions on submanifolds. We briefly mention some of the ideas and results.

\subsection{\label{LpRSECT} $L^p$ restriction theorems}

Theorem \ref{JJtheo1} required an improvement  on  the universal sup norm bounds 
 on the Cauchy data 
$$(\phi_j |_{\partial M}, \lambda_j^{-1} \partial_{\nu} \phi_j |_{\partial M} )$$
of Dirichlet  (resp.  Neumann)
eigenfunctions along the boundary. 
   We   denote by $r_q u$ the restriction of $u \in C(\bar{M})$ to $\partial M$ at the point
$q \in \partial M$, and we 
denote by  $\gamma_q^B$ the  boundary trace with boundary conditions $B$. Then let $\phi_j^b(q) = \gamma_q^B  \phi_j ,$ where \begin{equation}\label{SCTr}\;  \;\; \gamma_q^B =  \left\{ \begin{array}{ll}   r_q, & \rm{Neumann \;case} \\ &\\
r_q \partial_{\nu_q}, &  \rm{Dirichlet \;case} \end{array} \right. \end{equation}
Also let
\begin{equation} \label{BTSPEC} \Pi^b_{[0,\lambda]}(q,q') = \sum_{j: \lambda_j \leq \lambda} \phi_j^b(q) \phi_j^b(q') \end{equation} be  the boundary trace of the spectral projection for $\sqrt{-\Delta}$ for the interval $[0, \lambda]$.

\begin{prop} \label{PTWEYL} For any $C^{\infty}$ Riemannian manifold $(M, g)$ of dimension n with $C^{\infty}$ concave boundary $\partial M$,
$$\Pi^b_{[0,\lambda]}(q,q)= \left\{\begin{array}{ll} C_n \lambda^{n +2 } 
+  \lambda^2
R_D^b(\lambda, q) , & \mbox{Dirichlet}
\\ & \\
C_n \lambda^n 
+ R_N^b(\lambda, q), & \mbox{Neumann}.
\end{array} \right.$$

with

$$R^b_B(\lambda, q)  = O(\lambda^{n-1}) \;\; \mbox{uniformly in q}. $$

\end{prop}

Universal sup norm bounds  on boundary traces of  eigenfunctions are  obtained from
the jump in the remainder:

\begin{cor}\label{REMJUMP} Under the assumptions above,
\begin{equation}\label{eig3} \sum_{j: \lambda_j = \lambda} |\phi_j^b(q)|^2
=
R_B^b(\lambda,q)-R_B^b(\lambda-0,q) = O(\lambda^{n-1}),
\end{equation}
in the Neumann case and similarly with an extra factor of $\lambda^2$ in the 
Dirichlet case. The remainder is uniform in $q$. Hence, in the Neumann case,
$$\sup_{q \in \partial M} |\phi_j^b(q)| \leq C \lambda^{\frac{n-1}{2}}, \;\; $$
and similarly in the Dirichlet with the right side replaced by $\lambda^{\frac{n+1}{2}}.$
\end{cor}

\subsection{Manifolds with concave boundary and no self-focal boundary points never achieve maximal sup norm bounds}

In \cite{SoZ17} these  universal sup-norm estimates are improved in the case
case of manifolds of concave boundary and no self-focal points. 
We denote by $\Phi^t$ the  billiard flow (or broken geodesic flow) of $(M, g, \partial M)$. We also denote the broken
exponential map by
$\exp_x \xi = \pi \Phi^1(x, \xi)$. We refer to \cite{HoI-IV} (Chapter XXIV)   for background on these notions.

 Given any $x\in \bar{M}$, we denote by  $\lcal_x$ the set
of loop directions at $x$,
\begin{equation} \lcal_x = \{\xi \in S^*_x M : \exists T: \exp_x T
\xi = x \}.
\end{equation}

\begin{defin} We say that   $x$ is a self-focal  point if $|\lcal_x| > 0$ where
$|\cdot|_x$ denotes the surface measure on $S^*_x M$ determined by
the Euclidean metric $g_x$ on $T_x^*M$ induced by $g$. \end{defin}

.






\begin{theo} \label{SoZthm}
Let $(M, g)$ be a Riemannian manifold of dimension n with geodesically concave boundary. Suppose that
there exist no self-focal points $q \in \partial M$. Then, in the Neumann case,  $$\sup_{q \in \partial M} |\phi_j^b(q)| = o( \lambda^{\frac{n-1}{2}}), \;\; $$
and similarly in the Dirichlet with the right side replaced by $\lambda^{\frac{n+1}{2}}.$
\end{theo}
\bigskip

Thus, the Cauchy data can only achieve maximal sup norm bounds if there exists a self-focal point on the boundary. This is a sufficient sup-norm estimate to give Theorem \ref{JJtheo1}.  Recent results of J. Galkowski and others prove related results using microlocal defect measures (quantum limits).

\subsection{\label{KUZBDYSECT} Integrals of eigenfunctions over curves}

Another ingredient in Theorem \ref{JJtheo1}  is an estimate on integrals
$\int_{\beta} \phi_j ds$ of eigenfunctions over small arcs $\beta$ on the distinguished curve. The estimate originated in \cite{Zel92}, which showed
that $\int_H f \phi_j dS = O(1)$ for any compact Riemannian manifold and 
hypersurface. This universal estimate can be, and recently has been, improved in a series of articles, which can be used to improve the lower bound on numbers of nodal domains. In \cite{CGT17,Wy17a}, the bound 
$O(1)$ is improved to $o(1)$ under a non-focal condition on $H$, namely that the set of orthogonal geodesic arcs has measure zero. Logarithmic
improvements to $O(\frac{1}{(\log \lambda)^a})$ in negatively curved cases are proved in \cite{ChS15,SXZh17,XiZh16,Wy17b,Wy17c}.

\subsection{Log improvements} To prove the analogue of Theorem \ref{theo1b} in the boundary case requires log improvements on period bounds as in Theorem \ref{LOGKUZ} and log improvements on sup-norm bounds of Theorem \ref{SoZthm}. It seems to us that improvements are possible as long as the boundary has just one connected component. In this case, we can use the Melrose-Taylor parametrix in the exterior of one convex obstacle in a simply connected non-positively curved surface. It is much more complicated than the Hadamard parametrix in the boundary-less case, but it appears possible to prove analogues of the main estimates proved by B\'erard and others  in the boundary-less case.

\end{document}